\documentclass{article}

\usepackage[english]{babel}

\usepackage[letterpaper,top=2cm,bottom=2cm,left=3cm,right=3cm,marginparwidth=1.75cm]{geometry}

\usepackage{graphicx,xcolor} 
\usepackage{amsmath,amssymb,amsthm} 
\usepackage[utf8]{inputenc}
\usepackage{hyperref}
\usepackage[shortlabels]{enumitem}   
\usepackage{float}

\usepackage{graphicx}

\newtheorem{theorem}{Theorem}[section]
\newtheorem{corollary}[theorem]{Corollary}
\newtheorem{proposition}[theorem]{Proposition}
\newtheorem{lemma}[theorem]{Lemma}
\newtheorem{claim}{Claim}[theorem]
\newtheorem{subclaim}{Subclaim}[theorem]
\newtheorem{case}{Case}[theorem]
\newtheorem{subcase}{Subcase}[theorem]

\newtheorem{problem}[theorem]{Problem}

\newcommand{\floc}[1]{\textcolor{orange}{F: #1}}
\newcommand{\jbj}[1]{\textcolor{blue}{J: #1}}
\newcommand{\jbjc}[1]{\textcolor{magenta}{J: #1}}

\title{Complexity of (arc)-connectivity problems involving arc-reversals or deorientations\thanks{Research supported by the Independent Research Fond Denmark under grant number DFF
7014‐00037B.}}
\author{Jørgen Bang-Jensen\thanks{Department of Mathematics and Computer Science, University of Southern Denmark (email: jbj@imada.sdu.dk)}\and Florian Hörsch\thanks{CISPA, Saarbrücken , Germany (email: florian.hoersch@cispa.de)Most of this work was done while the author was a member of TU Ilmenau, Germany.}\and Matthias Kriesell\thanks{Department of mathematics Technische Universit\"at Ilmenau, Ilmenau, Germany (email:matthias.kriesell@tu-ilmenau.de )}}
\begin{document}
\maketitle

\begin{abstract}
  By a well known theorem of Robbins, a 
  graph $G$ has a strongly connected orientation if and only if $G$ is 2-edge-connected and it is easy to find, in linear time, either a cut edge of $G$ or a strong orientation of $G$. 
 A result of Durand de Gevigny shows that for every $k\geq 3$ it is NP-hard to decide if a given graph $G$ has a $k$-strong orientation. Thomassen showed that one can check in polynomial time whether a given graph has a 2-strong orientation. This implies that for a given digraph $D$ we can determine in polynomial time whether we can reorient (=reverse) some arcs of $D=(V,A)$ to obtain a 2-strong digraph $D'=(V,A')$. This naturally leads to the question of determining the minimum number of such arcs to reverse before the resulting graph is 2-strong. In this paper we show that finding this number is NP-hard.
  If a 2-connected graph $G$ has no 2-strong orientation, we may ask how many of its edges we may orient so that the resulting mixed graph is still 2-strong. Similarly, we may ask for a 2-edge-connected graph $G$ how many of its edges we can orient such that the resulting mixed graph remains 2-arc-strong.  We prove that when restricted to graphs satisfying suitable connectivity conditions, both of these problems are equivalent to finding the minimum number of edges we must double in a 2-edge-connected graph in order to obtain a 4-edge-connected graph. Using this, we show that all these three problems are NP-hard.

  Finally, we consider the operation of deorienting an arc $uv$ of a digraph $D$ meaning replacing it by an undirected edge between the same vertices. In terms of connectivity properties, this is equivalent to adding the opposite arc $vu$ to $D$. We prove that for every $\ell\geq 3$ it is NP-hard to find the minimum number of arcs to deorient in a digraph $D$ in order to obtain an $\ell$-strong digraph $D'$.
\end{abstract}
\noindent{\bf Keywords:} Arc-reversal, orientation, deorientation, vertex-connectivity, arc-connectivity

\section{Introduction}
In this paper graphs, digraphs and mixed graphs may have multiple edges and arcs but loops are not allowed. Note that a mixed graph can have  both edges and arcs. Here we also allow an arc and an edge to join the same pair of vertices. 
We use the notation $M=(V,E,A)$ to denote a mixed graph $M$ with edge set $E(M)=E$ and  arc set $A(M)=A$. We generally follow the notation of \cite{BG}.
For some positive integer $k$, a digraph $D=(V,A)$ is {\bf $\mathbf{k}$-strong} if it has  more than $k$ vertices and for any set $X$ of less than $k$ vertices, the digraph $D-X$ obtained by its removal is strongly connected. A digraph $D$ is {\bf $\mathbf{k}$-arc-strong} if every subdigraph obtained by deleting at most $k-1$ arcs from $D$ is strongly connected.

Let $a_k(D)$  denote the minimum number of new arcs one must add to a digraph $D=(V,A)$ in order to obtain a directed supergraph  $D'=(V,A\cup F)$ which is $k$-strong. As the complete digraph on $k+1$ vertices is $k$-strong we have $a_k(D)<\infty$ for every digraph on at least $k+1$ vertices. Frank and Jord\'an \cite{frankJCT65} found a characterization for $a_k(D)$ in terms of so-called one-way pairs as well as a polynomial algorithm for finding a minimum cardinality set $F$ of $a_k(D)$ new arcs whose addition to $D$ results in a $k$-strong digraph. The algorithm is exponential in $k$, a polynomial algorithm was given by V\'egh and Bencz\'ur in \cite{veghACMTA4}.

Similarly we let $\gamma_k(D)$ denote the minimum number of new arcs we need to add to $D=(V,A)$ so that the resulting digraph $D'=(V,A\cup F)$ is $k$-arc-strong. Frank \cite{frankSJDM5}
gave a characterization  for $\gamma_k(D)$ in terms of subpartitions of $V(D)$ and showed how to find a set of $\gamma_k(D)$ new arcs whose addition to $D$ results in a $k$-arc-strong digraph.

The following result, due to Nash-Williams, characterizes graphs that have a $k$-arc-strong orientation.

\begin{theorem}\cite{nashwilliamsCJM12}
  \label{thm:nashWthm}
  A graph $G=(V,E)$ has a  $k$-arc-strong orientation if and only if $G$ is $2k$-edge-connected.
\end{theorem}

Denote the underlying undirected graph of a digraph $D$ by $UG(D)$.
By Theorem \ref{thm:nashWthm}, it is possible to reorient some arcs of a digraph $D$ so that the resulting digraph is $k$-arc-strong if and only if $UG(D)$ is $2k$-edge-connected. Edmonds and Giles \cite{edmonds1977} showed how to use submodular flows to check in polynomial time whether  a given graph has a $k$-arc-strong orientation and find such an orientation when it exists. There is also a simple recursive algorithm based on the constructive proof of Theorem \ref{thm:nashWthm} using Lov\'asz's splitting theorem \cite[Exercise 6.53]{lovasz1979}. Using an algorithm or finding a minimum cost  feasible submodular flow \cite{edmonds1977,fujishigeMOR14} we can also determine the minimum number of arcs whose reversal in $D$ leads to a $k$-arc-strong reorientation.

Given the results above it  is natural to  ask for the complexity of determining the minimum number of arcs of a digraph $D$ that we need to  reverse to obtain a $k$-strong digraph.
Clearly $D$ has a set of arcs whose  reversal makes the new digraph $k$-strong if and only if its underlying undirected graph $UG(D)$ has a $k$-strong orientation, so we first consider the status of that problem.

As every $k$-strong digraph is also $k$-arc-strong, it follows from Theorem \ref{thm:nashWthm} that the following condition is necessary for a graph $G=(V,E)$ to have a $k$-strong orientation:
\begin{equation}
  \label{kstrongor}
\forall\hspace{1mm}X\subset V \mbox{ with } |X|<k \mbox{ the graph }G-X \mbox{ is }2(k-|X|) \mbox{-edge-connected.}
\end{equation}

For every fixed $k$ we can check whether a given graph $G$ satisfies (\ref{kstrongor}) by applying a polynomial number of applications of a polynomial algorithm for finding the edge-connectivity of a graph.
There are many such algorithms, see e.g. \cite{nagamochiSJDM5}.
Frank \cite{frank1995} conjectured that (\ref{kstrongor}) would also be sufficent. This was confirmed for $k=2$ by Thomassen \cite{thomassenJCT110}. 

\begin{theorem}\cite{thomassenJCT110}\label{thomassen}
    A graph $G=(V,E)$ has a 2-strong orientation if and only if it is 4-edge-connected and $G-v$ is 2-edge-connected for all $v \in V$.
\end{theorem}

Surprisingly $k=1,2$ are the only values of $k$ for which (\ref{kstrongor}) is sufficient. In fact 
Durand de Gevigney \cite{gevigneyJCT141} proved the following.

\begin{theorem}\cite{gevigneyJCT141}
  \label{thm:k>2NPC}
  For every $k\geq 3$ it is {\sf NP}-hard to decide whether a given input graph has a $k$-strong orientation.
\end{theorem}

 By Theorem \ref{thm:k>2NPC}, for $k\geq 3$
  it is already {\sf NP}-hard  to decide whether there is any set of arcs whose reversal makes a given digraph $k$-strong and hence a polynomial time algorithm for finding the minimum number of arcs to reverse in order to get a $k$-strong reorientation of a given $D$ is out of reach. The only remaining  case is $k=2$. By Thomassen's result and the fact that we can find the edge-connectivity of any graph in polynomial time,
  we can check in polynomial time whether a given digraph $D$ has a set of arcs whose reversal makes the resulting digraph 2-strong.

  Based on these observations, the first author asked at the conference ICGT 2022 in Montpellier whether  one could determine, for a given digraph $D=(V,A)$,  the minimum number of arcs whose reversal results in a 2-strong digraph.
In Section \ref{sec:reverse} we answer this question by proving that it is NP-hard to determine the minimum number of arcs whose reversal results in a 2-strong digraph.

  If a graph $G=(V,E)$ is 2-connected but does not satisfy (\ref{kstrongor}) for $k=2$, then it is natural to ask how many of its arcs we can orient so that the resulting mixed graph $M=(V,E',A)$ is still 2-strong. Here $A$ is the set of oriented edges (arcs) and $E'\subseteq E$ is the set of remaining edges that we did not orient. Similarly we can ask for a 2-edge-connected graph $G$ that is not 4-edge-connected and hence does not have a 2-arc-strong orientation by
  Theorem \ref{thm:nashWthm}, how many of its arcs we can orient so that the resulting mixed graph $M=(V,E',A)$ is still 2-arc-strong. In Section \ref{sec:partial}, we prove that for 2-connected graphs both of these problems are equivalent to the problem of finding the minimum number of edges in a 2-edge-connected graph one needs to double (add a copy of) in order to obtain a 4-edge-connected graph. We then prove that all of these problems are in fact NP-hard.

  In Section \ref{sec:deor} we consider another operation for increasing the (arc)-connectivity of a digraph, namely that is deorienting arcs. By {\bf deorienting} an arc $uv$ we mean the operation of replacing the arc by an undirected edge between $u$ and $v$. Note that the effect of this operation on the connectivity properties of the digraph is the same as when adding a (copy of) the opposite arc $vu$. Deorienting a subset of the  arcs of a digraph can increase its (arc)-connectivity and it corresponds to a variation of the (arc)-connectivity augmentation problems that we discussed in the beginning of the paper. We prove  that for every $k\geq 3$ it is NP-hard to find the minimum number of arcs one needs to deorient in a given digraph $D$ in order to obtain a $k$-strong mixed graph. This partially answers a question raised in \cite{BG}. The complexity of the analogous problem  to find the minimum number of arcs one needs to deorient in a given digraph $D$ in order to obtain a $k$-arc-strong mixed graph is unknown. We point out that there is a 2-approximation algorithm for the problem and show that the problem is NP-hard if we wish to achieve given local arc-connectivities. 

\section{Arc reversals}\label{sec:reverse}
Given a digraph $D=(V,A)$, by {\bf reversing (reorienting)} an arc $a \in A$, we mean the operation of exchanging the head and the tail of $a$. We wish to understand how many arcs we need to reorient in a given digraph in order to obtain a digraph that satisfies some prescribed connectivity condition. 
As mentioned in the introduction, using algorithms for minimum cost submodular flows, we can determine in polynomial time for every positive integer $k$ and digraph $D$ whose underlying graph is $2k$-edge-connected, the minimum number of arcs we need to reorient in a given digraph to obtain a $k$-arc-connected digraph.

For vertex-connectivity, the following is an immediate consequence of Theorem \ref{thm:k>2NPC}.
\begin{theorem}
For every integer $k\geq 3$, it is NP-hard to compute the minimum number of arcs we need to reorient in a given digraph to obtain a $k$-strong digraph.
\end{theorem}  

As a digraph is strong if and only if it is 1-arc-strong,  it follows from our remarks 
in the introduction that we can determine the minimum number of arcs whose reversal results in a strong digraph 
in polynomial time.
We deal with the remaining open case, namely 2-strong digraphs. Formally, we consider the following problem:
\begin{center}
\begin{tabular}{|ccc|}\hline
 & \begin{minipage}{14cm}
\vspace{2mm}

{\bf Minimum 2-Strong Arc Reversal  (M2SAR):}
\vspace{2mm}

{\bf Input:} A digraph $D$, an integer $k$.

\vspace{2mm}

{\bf Question:} Is there a 2-strong digraph $D'$ which can be obtained from $D$ by reversing at most $k$ arcs?

\vspace{2mm}
\end{minipage} & \\ \hline
\end{tabular}
\end{center}

The following theorem completes the picture in the above discussion.

\begin{theorem}\label{thm:2reverse}
It is NP-hard to compute the minimum number of arcs we need to reorient in a given digraph to obtain a $2$-strong digraph.
\end{theorem}

The rest of this section is concerned with proving Theorem \ref{thm:2reverse}. In Section \ref{gadg}, we describe a gadget that will prove useful. In Section \ref{mainred1}, we give the main reduction proving Theorem \ref{thm:2reverse}.
\subsection{The Gadget}\label{gadg}


We now describe a gadget we need in our reduction. For four distinct vertices $x_0,y_0,z_0,v^*$ and a positive integer $k$, an {\it$(x_0,y_0,z_0,v^*,k)$-out-rocket} is a digraph $R$ with $V(R)=\{x_0,\ldots,x_k,y_0,\ldots,y_k,z_0,\ldots,z_k,u,v^*\}$ and that contains the following arcs:

\begin{itemize}
\item the arcs $x_iy_i,y_iz_i$ and $z_ix_i$ for $i=1,\ldots,k$,
\item the arcs $x_ix_{i+1},y_iy_{i+1}$ and $z_{i+1}z_i$ for $i=0,\ldots,k-1$,
\item the arcs $x_ku,y_ku,uz_k$ and $uv^*$.
\end{itemize}

An illustration can be found in Figure \ref{rocket}. 

\begin{figure}[h]
    \centering
        \includegraphics[width=.2\textwidth]{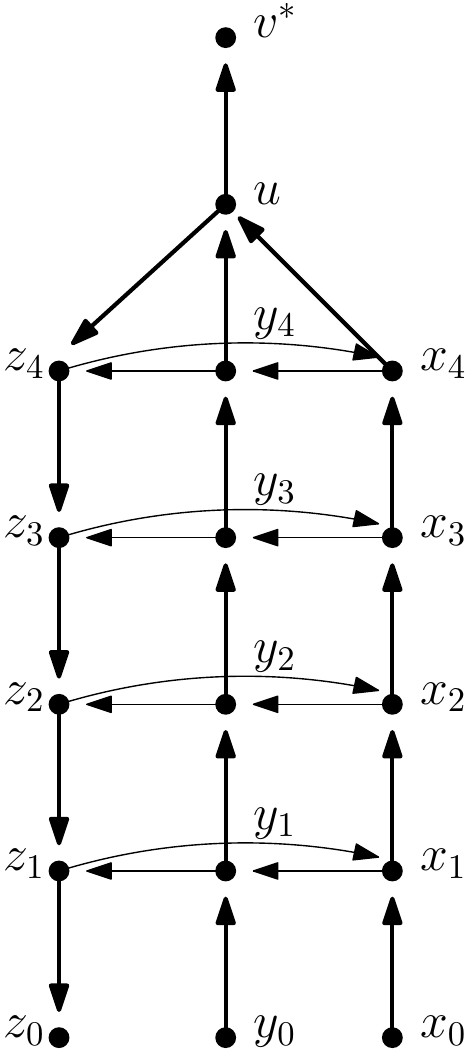}
        \caption{A $(x_0,y_0,z_0,v^*,4)$-out-rocket}\label{rocket}
\end{figure}

We call $uv^*$ the {\bf tip arc} of $R$. An {\it$(x_0,y_0,z_0,v^*,k)$-in-rocket} is obtained from $R$ by reversing all arcs. We call $x_0,y_0,z_0$ and $v^*$ the {\bf exterior} vertices and the remaining vertices of $V(R)$ the {\bf interior} vertices of $R$. The following two observations on the connectivity properties of rockets in which $R$ is an $(x_0,y_0,z_0,v^*,k)$-in-rocket or an $(x_0,y_0,z_0,v^*,k)$-out-rocket are easy to verify and hence given without proof.


\begin{proposition}\label{triv2}
Let $R'$ be obtained from $R$ by deleting an exterior vertex and identifying  the three remaining exterior vertices. Then $R'$ is strongly connected.
\end{proposition}

\begin{proposition}\label{triv3}
Let $R'$ be obtained from $R$ by identifying  the four exterior exterior vertices. Then $R'-x$ is strongly connected for all interior vertices $x \in V(R)$.
\end{proposition}

In the following, we say that a digraph $D$ {\bf contains} an $(x_0,y_0,z_0,v^*,k)$-out-rocket $R$ or an $(x_0,y_0,z_0,v^*,k)$-in-rocket $R$ if $D$ contains $R$ as a subgraph and $\delta_{D}^+(v)=\delta_{R}^+(v)$ and $\delta_{D}^-(v)=\delta_{R}^-(v)$ hold for all interior vertices $v \in V(R)$.  

The following is the crucial property of rockets.

\begin{lemma}\label{tip}
Let $D_1,D_2$ be digraphs with $UG(D_1)=UG(D_2)$ such that $D_1$ contains a $(x_0,y_0,z_0,v^*,k)$-out-rocket $R$ or a $(x_0,y_0,z_0,v^*,k)$-in-rocket $R$ for some $x_0,y_0,z_0,v^*\in V(D_1)$ and some positive integer $k$, $D_2$ is 2-strong and the tip arc of $R$ is reversed in $D_2$. Then the number of arcs in $A(D_1)$ that are reversed in $D_2$ is at least $k+1$.
\end{lemma}
\begin{proof} By symmetry, we may suppose that $R$ is a $(x_0,y_0,z_0,v^*,k)$-out-rocket.
Denote the vertices of $R$ by $x_0,\ldots,x_k,y_0,\ldots,y_k,z_0,\ldots,z_k,u,v^*$ like in the definition of a rocket. By assumption, the tip arc of $R$ is reversed in $D_2$. It hence suffices to prove that for every $i=0,\ldots,k-1$, one of the arcs $x_ix_{i+1}$ and $y_iy_{i+1}$ is reversed in $D_2$. Fix some $i \in \{0,\ldots,k-1\}$ and let $Z=\{x_{i+1},\ldots,x_k,y_{i+1},\ldots,y_k,z_{i+1},\ldots,z_k,u\}$. Observe that $d_{UG(D_2)}(Z)=4$ and hence, as $D_2$ is 2-strong, we obtain $d_{D_2}^-(Z)=2$. As $v^*u$ enters $Z$ in $D_2$, at most one of the arcs $x_ix_{i+1}$ and $y_iy_{i+1}$ can enter $Z$ in $D_2$. This finishes the proof.
\end{proof}

\subsection{The reduction}\label{mainred1}

We here give the main reduction proving Theorem \ref{thm:2reverse}.
For our reduction, we need the following problem where an orientation of a mixed graph $M=(V,E,A)$ is any digraph the can be obtained from $M$ by assigning an orientation to each edge of $M$.

\begin{center}
\begin{tabular}{|ccc|}\hline
 & \begin{minipage}{14cm}
\vspace{2mm}

{\bf Independent 2-strong orientation of mixed graphs  (I2VCOMG):}
\vspace{2mm}

{\bf Input:} A mixed graph $M$, a set $T \subseteq V(M)$ that is independent in $UG(M)$.

\vspace{2mm}

{\bf Question:} Is there a 2-arc-strong  orientation $\vec{M}$ of $M$ such that $\vec{M}-v$ is strongly connected for all $v \in T$?

\vspace{2mm}
\end{minipage} & \\ \hline
\end{tabular}
\end{center}

The following is implicitely proven in \cite{hs}.

\begin{theorem}\label{alt}
I2VCOMG is NP-hard.
\end{theorem}

For a mixed graph $M=(V,E,A)$  and $X \subseteq V$, we use {\bf $\delta_M(X),\delta^+_M(X),\delta^-_M(X)$} to denote respectively, the set of edges of $E$ with one end vertex in $X$, the set of arcs of $A$ with leaving $X$ and the set of arcs of $A$ entering $X$. Further we use ${\bf d_M(X)}=|\delta_M(X)|, {\bf d_M^+(X)}=|\delta_M^+(X)|$ and ${\bf d_M^-(X)}=|\delta_M^-(X)|$.
We are now ready to prove Theorem \ref{thm:2reverse} through a reduction from I2VCOMG. 
\begin{proof}(of Theorem \ref{thm:2reverse}) Let $(M,T)$ be an instance of I2VCOMG.
For every arc $a \in A(M)$, we choose one vertex $v_a \in \{head(a),tail(a)\}-T$. Observe that such a vertex always exists as $T$ is an independent set in $UG(M)$.

We now create a digraph $D$. First, we let $V(D)$ contain $T$. Further, for every $v \in V(M)-T$, we let $V(D)$ contain a set $X_v$ that contains
\begin{itemize}
\item a vertex $x^{v,e}$ for every edge $e \in \delta_M(v)$,
\item a vertex $x^{v,a}$ for every arc $a \in \delta_M^+(v) \cup \delta_M^-(v)$ with $v_a \neq v$,
\item three vertices $x_0^{v,a},y_0^{v,a},z_0^{v,a}$ for every $a \in \delta_M^+(v) \cup \delta_M^-(v)$ with $v_a = v$.
\end{itemize}

For every $t \in T$ and $a \in \delta_M(t)\cup  \delta_M^+(t)\cup  \delta_M^-(t)$, for convenience, we also refer to $t$ by $x^{t,a}$. For every $v \in V(M)-T$, we let $D[X_v]$ be a biclique. Next, for every $e=uv \in E(M)$, we add an arc in an arbitrary direction linking $x^{u,e}$ and $x^{v,e}$. For every $a =uv \in A(M)$ with $v_a=u$, we let $D$ contain a $(x^{u,a}_0,y^{u,a}_0,z^{u,a}_0,x^{v,a},|E(M)|)$-out-rocket $R_a$ and for every $a =uv \in A(M)$ with $v_a=v$, we let $D$ contain a $(x^{v,a}_0,y^{v,a}_0,z^{v,a}_0,x^{u,a},|E(M)|)$-in-rocket $R_a$. This finishes the description of $D$. For every $v \in V(M)-T$, we let $Y_v=X_v \cup \bigcup \limits_{\substack{a \in \delta_M^+(v)\cup \delta_M^-(v)\\v_a=v}}V(R_a)-\bigcup \limits_{w \in V(M)-T-v}X_w-T$. For an illustration, see Figure \ref{example}.\\


\begin{figure}[h]
    \centering
        \includegraphics[width=1\textwidth]{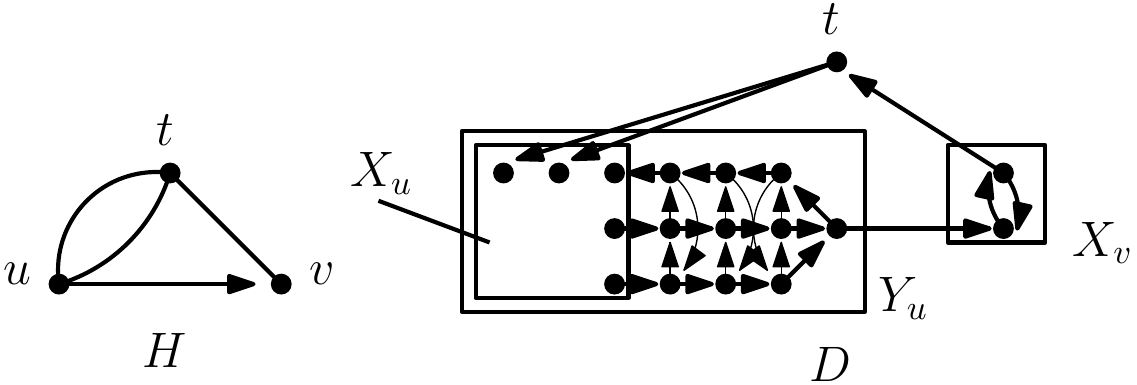}
        \caption{An example illustrating the construction. The instance consists of the mixed graph $M$ and $T=\{t\}$ and we choose $v_a=u$ for the arc $a=uv$. The arcs of the biclique in $D[X_u]$ have been omitted.}\label{example}
\end{figure}
\smallskip

We now show that $(D,|E(M)|)$ is a positive instance of M2SAR if and only if $(M,T)$ is a positive instance of I2VCOMG.
\smallskip

First suppose that $(D,|E(M)|)$ is a positive instance of M2SAR, so there is a 2-strong digraph $D_1$ that is obtained from $D$ by at most $|E(M)|$ arc reversals. We now create an orientation $\vec{M}$ of $M$. For every $e=uv \in E(M)$, we orient $e$ from $u$ to $v$ if and only if $D_1$ contains an arc from $x^{e,u}$ to $x^{e,v}$. 

In order to see that $\vec{M}$ is 2-arc-connected, consider some $Z \subseteq V(M)$ and let $Z'=\bigcup_{z \in Z-T}Y_z\cup(Z \cap T)$. Let $a \in \delta_{D_1}^+(Z')$. If $a$ is of the form $x^{u,e}x^{v,e}$ for some $e \in E(M)$, then $uv \in \delta_{\vec{M}}^+(Z)$ by construction. Otherwise, the arc corresponding to $a$ in $D$ is the tip arc of a rocket $R_{a'}$ in $D$. It follows from Lemma \ref{tip} that $a$ also exists in $D$ and hence $a' \in \delta_{\vec{M}}^+(Z)$ by construction. As $D_1$ is 2-strong, it follows that $d_{\vec{M}}^+(Z)\geq d_{D_1}^+(Z')\geq 2$, so $\vec{M}$ is 2-arc-connected.

Now consider $\vec{M}-t$ for some $t \in T$. Let $Z \subseteq V(M-t)$ and $Z'=\bigcup_{z \in Z}Y_z$. As $D_1-t$ is strongly connected, a similar argument as before shows that  $d_{\vec{M}-t}^+(Z)\geq d_{D_1-t}^+(Z')\geq 1$, so $\vec{M}-t$ is strongly connected.
\smallskip

Now suppose that $(M,T)$ is a positive instance of I2VCOMG, so there is a 2-arc-connected orientation $\vec{M}$ of $M$ such that $\vec{M}-t$ is strongly connected for all $t \in T$. We now obtain $D_1$ from $D$ by reversing all the arcs of the form $x^{u,e}x^{v,e}$ for some $e \in E(M)$ for which the edge $e$ is oriented from $v$ to $u$ in $\vec{M}$. Observe that $D_1$ is obtained from $D$ by reversing at most $|E(M)|$ arcs. 

We still need to show that $D_1-x$ is strongly connected for all $x \in V(D_1)$. We distinguish three cases.

\begin{case}
$x \in T$.
\end{case}
By Proposition \ref{triv2} and as $G[X_v]$ is a biclique, $D_1[Y_v]$ is strongly connected for all $v \in V(D_1)-T$. It hence suffices to prove that the graph obtained from $D_1$ by contracting $Y_v$ into a single vertex for all $v \in V(D_1)-T$ is strongly connected. This graph is isomorphic to $\vec{M}-x$ and hence strongly connected by assumption.

\begin{case}
$x \in V(R_a)-T$ for some $a \in A(M)$.
\end{case}
In this case, there is a unique $v_0 \in V(M)-T$ such that $x \in Y_{v_0}$. As $\vec{M}$ is 2-arc-connected, we obtain that $\vec{M}-a$ is strongly connected. Further, by Proposition \ref{triv2} and as $D_1[X_v]$ is a biclique for all $v \in V$, we obtain that $D_1[Y_v]$ is strongly connected for all $v \in V(M)-T-v_0$ and $D_1[Y_{v_0}-V(R_a)]$ is strongly connected. We obtain that $D_1-V(R_a)$ is strongly connected. As $X_v$ is a biclique for all $v \in V(M)-T$, all the exterior vertices of $R_a$ which are distinct from $x$ are in the same connected component of $D_1-x$ as $V(D_1)-V(R_a)$. We obtain that $D_1-x$ is strongly connected by Propositions \ref{triv2} and \ref{triv3}.

\begin{case}
$x = x^{u_0,e}$ for some $e=u_0v_0 \in E(M)$.
\end{case}

By Proposition \ref{triv2} and as $D_1[X_v]$ is a biclique for all $v \in V$, the subdigraph  $D_1[Y_v]$ is strongly connected for all $v \in V(M)-T-u_0$ and $D_1[Y_{u_0}-x]$ is strongly connected. As $\vec{M}$ is 2-arc-strong, $\vec{M}-\vec{e}$ is strongly connected. Now it follows from the way we constructed $D$ from $M$ that $D_1-x$ is strongly connected. 
\end{proof}
We wish to remark that a slight modification of this reduction shows that the minimization problem associated to M2SAR does not admit an $\alpha$-approximation algorithm for any constant $\alpha$.
Further, similar results can be obtained when restricting the input graphs to being acyclic.

\section{Partial orientations}\label{sec:partial}
We say that a mixed graph $M=(V,E,A)$ is {\bf $\mathbf{k}$-arc-strong} if $d_A^+(X)+d_E(X)\geq k$ for every non-empty proper subset $X$ of $V$. Here $d_A^+(X)$ denotes the number of arcs leaving $X$ in the subdigraph $D$ of $M$ induced by the arcs in $A$ and $d_E(X)$ denotes the number of edges with exactly one end in $X$ in the subgraph $G$ of $M$ induced by the edges in $E$. Similarly, a mixed graph $M=(V,E,A)$ is $k$-strong if it has more than $k$ vertices and deleting any set of less than $k$ vertices from $M$ leaves a 1-arc-strong (strong) mixed graph.
The next two observations are easy to prove.

\begin{proposition}\label{undigon}
   Let  $M_1,M_2$ be mixed graphs such that $M_2$ is obtained from $M_1$ by replacing a digon by an undirected edge between the same two vertices and let $k$ be a positive integer. Then $M_2$ is $k$-arc-strong if and only if $M_1$ is $k$-arc-strong.
\end{proposition}
\begin{proposition}\label{undigon1}
   Let  $M_1,M_2$ be mixed graphs such that $M_2$ is obtained from $M_1$ by replacing a digon by an undirected edge between the same two vertices and let $k$ be a positive integer. Then $M_2$ is $k$-strong if and only if $M_1$ is $k$-strong.
\end{proposition}

A {\bf partial orientation} of an undirected graph is a mixed graph that is obtained from orienting some of the edges in the graph. The following is the central question of this section:
Given a graph $G=(V,E)$ and an integer $k$, can we find a set  $F\subseteq E$ with $|F|\geq k$ and an orientation $\vec{F}$ of $F$ such that the mixed graph 
$M=(V,E-F,\vec{F})$ satisfies a certain prescribed connectivity property?\\

The case of strong connectivity is solvable in polynomial time due to the following easy consequence of the restriction of Theorem \ref{thm:nashWthm} to $k=1$ which was proved earlier by Robbins \cite{robb}.
\begin{theorem}\cite{robb}
Let $G=(V,E)$ be a graph and $k$ a positive integer. Then there exists a strongly connected partial orientation of $G$ in which $k$ arcs are oriented if and only if $G$ is connected and $k \leq |E|-b(G)$ where $b(G)$ denotes the number of bridges of $G$.
\end{theorem}
The following is an immediate consequence of Theorem \ref{thm:k>2NPC}.

\begin{theorem}
For any $\ell \geq 3$, it is NP-hard to decide whether there exists an $\ell$-strong partial orientation of a given graph $G$ in which at least $k$ edges are oriented where $k$ is part of the input. 
\end{theorem}

By these two  results, the only remaining case for vertex-connectivity is $\ell=2$. We are going to show that this case is NP-hard and that the same holds for 2-arc-connectivity.


The following results describe a close relationship between the above mentioned problem on partial orientations and the problem of making a graph satisfy certain connectivity properties by doubling some edges.
We need the following result which is a direct consequence of Corollary 2 in \cite{kiralyJCT96}.
\begin{proposition}\label{oppo}
    Let $G$ be a graph that has a $k$-arc-connected orientation for some positive integer $k$ and let $(e_1,f_1),\ldots,(e_t,f_t)$ be a collection of pairwise disjoint pairs of parallel edges in $G$. Then $G$ has a $k$-arc-connected orientation in which $e_i$ and $f_i$ are oriented in opposite directions for $i=1,\ldots,t$.
\end{proposition}

\begin{lemma}\label{lem:one}
    Let $G$ be a 2-edge-connected graph and $k$ a positive integer. Then $G$ can be made 4-edge-connected by doubling at most $k$ edges if and only if $G$ has a 2-arc-strong partial orientation in which at most $k$ edges remain unoriented.
\end{lemma} 
\begin{proof}
    First suppose that a 4-edge-connected graph $G'=(V,E')$ can be obtained from $G$ by doubling a set $F$ of at most $k$ edges. By Theorem \ref{thm:nashWthm}, there is a 2-arc-connected orientation $\vec{G'}$ of $G'$. Further, by Proposition \ref{oppo}, we may assume that for every $e \in F$, the two edges in $E'$ corresponding to $e$ are oriented in opposite directions in $\vec{G'}$. Now let $M$ be the mixed graph in which each of these pairs is replaced by a single undirected edge. By Proposition \ref{undigon}, we obtain that $M$ is 2-arc-strong. Further, $M$ is a partial orientation of $G$ in which only the edges of $F$, hence at most $k$ edges, remain unoriented.

    Now suppose that there is a 2-arc-strong partial orientation $M$ of $G$ in which the set $F$ of undirected edges is of size  at most $k$. Let $\vec{G'}$ be obtained from $M$ by replacing every undirected edge by a digon. By Proposition \ref{undigon}, we obtain that $\vec{G'}$ is 
    2-arc-strong. Let $G'$ be the underlying graph of $\vec{G'}$. By Theorem \ref{thm:nashWthm}, we obtain that $G'$ is 4-edge-connected. Further, $G'$ is obtained from $G$ by doubling the edges in $F$, hence at most $k$ edges.
\end{proof}

\begin{lemma}\label{lem:two}
    Let $G=(V,E)$ be a 2-vertex-connected graph and $k$ a positive integer. Then there is a 4-edge-connected graph $G'$ that can be obtained from $G$ by doubling at most $k$ edges for which $G'-v$ is 2-edge-connected for every $v\in V$ if and only if $G$ has a 2-strong partial orientation in which at most $k$ edges remain unoriented.
\end{lemma} 
\begin{proof}
    First suppose that by doubling a set $F$ of at most $k$ edges of $G$ we can obtain  a 4-edge-connected graph $G'=(V,E')$ for which $G'-v$ is 2-edge-connected for every $v\in V$. By Theorem \ref{thomassen}, there is a 2-strong orientation $\vec{G'}$ of $G'$. Further,  we may clearly assume that for every $e \in F$, the two edges in $E'$ corresponding to $e$ are oriented in opposite directions. Now let $M$ be the mixed graph in which each of these pairs is replaced by a single undirected edge. By Proposition \ref{undigon1}, we obtain that $M$ is 2-strong. Further, $M$ is a partial orientation of $G$ in which only the edges of $F$, hence at most $k$ edges, remain unoriented.

    Now suppose that there is a 2-strong partial orientation $M$ of $G$ in which the set $F$ of undirected edges is of size  at most $k$. Let $\vec{G'}$ be obtained from $M$ by replacing every undirected edge by a digon. By Proposition \ref{undigon1}, we obtain that $\vec{G'}$ is 2-strong. Let $G'$ be the underlying graph of $\vec{G'}$. By Theorem \ref{thm:nashWthm}, we obtain that $G'$ is 4-edge-connected and since $\vec{G'}-v$ is strong for all $v \in V$, we get that $G'-v$ is 2-edge-connected. Further, $G'$ is obtained from $G$ by doubling the edges in $F$, hence at most $k$ edges.
\end{proof}

The rest of this section is structured as follows: In Section \ref{sec34}, we prove that the problem of doubling the minimum number of edges of a graph to obtain a 4-edge-connected graph is NP-hard. As this result also holds for the graph classes considered in Lemmas \ref{lem:one} and \ref{lem:two}, we obtain hardness results for the corresponding partial orientation problems. In Section \ref{approx34}, as a second application of Lemmas \ref{lem:one} and \ref{lem:two}, we obtain approximation algorithms for the partial orientation problems in consideration, relying on a result of Cecchetto, Traub and Zenklusen \cite{ctz}. Motivated by the result in Section \ref{sec34}, we study the problem of making a graph 3-edge-connected by doubling edges in Section \ref{sec23} and show that this problem can be solved in polynomial time.

\subsection{4-edge-connectivity augmentation by doubling edges}\label{sec34}
Formally, we consider the following problem where the choice of properties of the input graph $H$ is motivated by Lemmas \ref{lem:one} and \ref{lem:two}:

\medskip

\begin{tabular}{|ccc|}\hline
 & \begin{minipage}{14cm}
\vspace{2mm}

\noindent{}{\bf 4 Edge-Doubling Augmentation (4EDA)}
\vspace{2mm}

{\bf Input:} A  graph $H$ such that $H-v$ is 2-edge-connected for all $v \in V(H)$ and an integer $k$.

\vspace{2mm}

{\bf Question:} Can $H$ be made 4-edge-connected by doubling at most $k$ edges?

\vspace{2mm}

\end{minipage} & \\ \hline
\end{tabular}

\vspace{4mm}

The following is the main result of this section.
\begin{theorem}\label{34hard}
4EDA is NP-hard.
\end{theorem}
Together with Lemmas \ref{lem:one} and \ref{lem:two}, Theorem \ref{34hard} immediately implies the following results for partial orientations.

\begin{corollary}
    Given a graph $G$ and a positive integer $k$, it is NP-hard to decide whether there is a $2$-arc-strong partial orientation of $G$ in which at least $k$ arcs are oriented.
\end{corollary}

\begin{corollary}
    Given a graph $G$ and a positive integer $k$, it is NP-hard to decide whether there is a $2$-strong partial orientation of $G$ in which at least $k$ arcs are oriented.
\end{corollary}

The rest of Section \ref{sec34} is concerned with proving Theorem \ref{34hard} by a reduction from a variation of vertex cover. In Section \ref{vcprel}, we introduce this variation of vertex cover and show that it remains hard indeed. In Section \ref{constr}, we describe our construction and prove some of its important properties. In Section \ref{demo}, we show that the reduction works indeed.

\subsubsection{Preliminaries on vertex cover}\label{vcprel}

The vertex cover problem can be described as follows:

\medskip

\begin{tabular}{|ccc|}\hline
 & \begin{minipage}{14cm}
\vspace{2mm}

\noindent{}{\bf Vertex Cover (VC)}
\vspace{2mm}

{\bf Input:} A graph $G$, an integer $k$.

\vspace{2mm}

{\bf Question:} Is there a set $S \subseteq V(G)$ with $|S|\leq k$ such that $S$ contains at least one endvertex of $e$ for all $e \in E(G)$?

\vspace{2mm}

\end{minipage} & \\ \hline
\end{tabular}

\vspace{4mm}
The following result is well-known.
\begin{theorem}\label{vcdure}\cite{gj}
VC is NP-hard for cubic 2-vertex-connected graphs.
\end{theorem}

Let $\mathcal{G}$ be the class of graphs that arise from a cubic 2-vertex-connected graph by subdividing every edge twice.

\begin{proposition}\label{vcgdure}
VC is NP-hard for graphs in $\mathcal{G}$.
\end{proposition}
\begin{proof}
Let $G$ be a 2-connected, cubic graph, $k$ a positive integer and $G'$ the graph which arises from $G$ by subdividing every edge twice. By definition, we have $G' \in \mathcal{G}$. We now show that $(G',k+|E(G)|)$ is a positive instance of VC if and only if $(G,k)$ is a positive instance of VC.

First suppose that $(G,k)$ is a positive instance of VC, so there is a vertex cover $S \subseteq V(G)$ with $|S|\leq k$. By definition, for every $e=uv\in E(G)$, at least one of $u$ and $v$ is in $S$. Hence at most one of the two subdivision vertices of $e$ does not have a neighbor in $S$ in $G'$. We create $S'$ by adding this vertex to $S$, choosing an arbitrary one of the two subdivision vertices if both $u$ and $v$ are contained in $S$. It is easy to see that $S'$ is a vertex cover of $G'$ and satisfies $|S'|=|E(G)|+|S|\leq |E(G)|+k$.

Now suppose that $(G',k+|E(G)|)$ is a positive instance of VC, so there is a vertex cover $S' \subseteq V(G')$ with $|S'|\leq k+|E(G)|$. We may suppose that $S'$ is chosen so that the number of edges $e \in E(G)$ such that $S'$ contains both subdivision vertices of $e$ is minimized. Suppose that $S'$ contains both subdivision vertices of some edge $e=uv \in E(G)$. Let $S''$ be obtained from $S'$ by deleting the subdivision vertex of $e$ which is a neighbor of $u$ and adding $u$ if it is not yet contained. Then $S''$ is a vertex cover of $G'$ of at most the same size as $S'$, a contradiction to the choice of $S'$. Hence for every $e \in E(G)$, $S'$ contains at most one of the subdivision vertices of $e$. Now let $S=S'\cap V(G)$. It is easy to see that $S$ is a vertex cover of $G$ that satisfies $|S|=|S'|-|E(G)|\leq k$. This finishes the proof by Theorem \ref{vcdure}.
\end{proof}

For a graph $G \in \mathcal{G}$, a {\bf legal} path decomposition of $G$ is a set of subpaths $\mathcal{P}=\mathcal{P}_1 \cup \mathcal{P}_2$ of $G$ with the following properties:

\begin{itemize}
    \item $\{E(P):P \in \mathcal{P}\}$ is a partition of $E(G)$,
    \item every $P \in \mathcal{P}_i$ contains exactly $i$ edges for $i=1,2$,
    \item every $v \in V(G)$ is contained in exactly two paths of $\mathcal{P}$.
\end{itemize}

\begin{proposition}\label{leg}
Every $G \in \mathcal{G}$ has a legal path decomposition.
\end{proposition}
\begin{proof}
For every vertex $v$ of $G$ of degree 3, choose two arbitrary edges of $G$ which are incident with $v$ and add the corresponding path to $\mathcal{P}_2$. For all remaining edges of $G$, add the path that contains only this edge to $\mathcal{P}_1$. It is easy to see that $\mathcal{P}_1\cup \mathcal{P}_2$ has the desired properties.
\end{proof}

\begin{proposition}\label{trivial}
Let $G \in \mathcal{G}$ and let $\mathcal{P}=\mathcal{P}_1 \cup \mathcal{P}_2$ be a legal path decomposition of $G$. Then $|\mathcal{P}_1|=\frac{5}{8}|V(G)|, |\mathcal{P}_2|=\frac{1}{4}|V(G)|$ and every vertex is contained in at most one path of $\mathcal{P}_2$.
\end{proposition}
\begin{proof}
Let $G'$ be the graph from which $G$ is obtained by subdividing every edge twice. Then, by the last property of legal decompositions, we obtain that every $v\in V(G')$ is the middle vertex of exactly one path in $\mathcal{P}_2$ and no vertex in $V(G)-V(G')$ is the middle vertex of a path in $\mathcal{P}_2$. We obtain $|\mathcal{P}_2|=|V(G')|=\frac{1}{4}|V(G)|$. We further have $|\mathcal{P}_1|=|E(G)|-2|\mathcal{P}_2|=3|E(G')|-2|\mathcal{P}_2|=\frac{9}{2}|V(G')|-2|V(G')|=\frac{5}{2}|V(G')|=\frac{5}{8}|V(G)|$. The fact that the middle vertex of every path in $\mathcal{P}_2$ is in $V(G')$ implies the second property.
\end{proof}
\subsubsection{The construction and the main lemma}\label{constr}

This section contains the first part of the proof of Theorem \ref{34hard}. Based on Proposition \ref{vcgdure}, we proceed by a reduction from VC with the additional assumption that the input graph is in $\mathcal{G}$. Let $(G,k)$ be an instance of VC with $G \in \mathcal{G}$. Clearly, we may suppose that $|V(G)|\geq 5$. By Proposition \ref{leg}, there is a legal path decomposition $\mathcal{P}=\mathcal{P}_1 \cup \mathcal{P}_2$ of $G$ which can easily be computed in polynomial time. We now create a graph $H$. First, for every $P \in \mathcal{P}_1$, we let $V(H)$ contain a vertex $x_P$. For every $P=uvw \in \mathcal{P}_2$, we let $H$ contain a path gadget with vertex set $X_P=\{x_P^{u},x_P^{v},x_P^{w},x_P^1,\ldots,x_P^8\}$ and edges $x_P^{u}x_P^2,x_P^{u}x_P^3,x_P^{u}x_P^4,x_P^{u}x_P^5,x_P^{v}x_P^1,x_P^{v}x_P^8,x_P^{w}x_P^7,x_P^{w}x_P^8,x_P^{1}x_P^2,x_P^{1}x_P^8,x_P^{2}x_P^3,x_P^{3}x_P^4,x_P^{4}x_P^5,x_P^{5}x_P^6,\\
x_P^{6}x_P^7,x_P^{6}x_P^8,x_P^{7}x_P^8$. Observe that the roles of $u$ and $w$ could be exchanged in this construction. However, this ambiguity will be of no effect. An illustration can be found in Figure \ref{ajouter!}.

\begin{figure}[h]
    \centering
        \includegraphics[width=.4\textwidth]{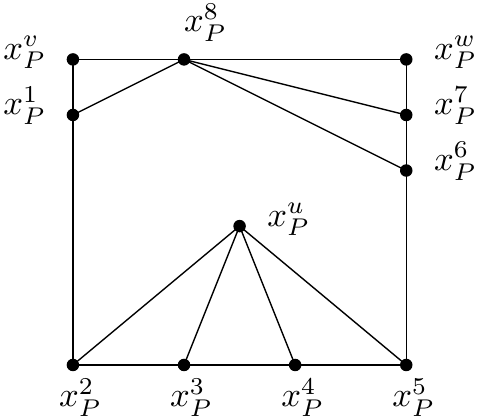}
        \caption{An example for a gadget for a path $P$ 
        in $\mathcal{P}_2$.}\label{ajouter!}
\end{figure}

Next, we add a vertex $y$ and an edge linking $y$ and $x_P$ for all $P \in \mathcal{P}_1$. Finally, for every $v \in V(G)$ that is contained in two paths $P,P'\in \mathcal{P}_1$, we add an edge $e_v$ linking $x_P$ and $x_{P'}$ and for every $v \in V(G)$ that is contained in a path $P\in \mathcal{P}_1$ and a path $P'\in \mathcal{P}_2$, we add an edge $e_v$ linking $x_P$ and $x_{P'}^v$. Observe that, as $\mathcal{P}$ is legal and by Proposition \ref{trivial}, this operation is well-defined and we have added exactly one edge $e_v$ for every $v \in V(G)$. This finishes the description of $H$. An illustration can be found in Figure \ref{ajouteraussi}. 

\begin{figure}[H]
    \centering
        \includegraphics[width=.7\textwidth]{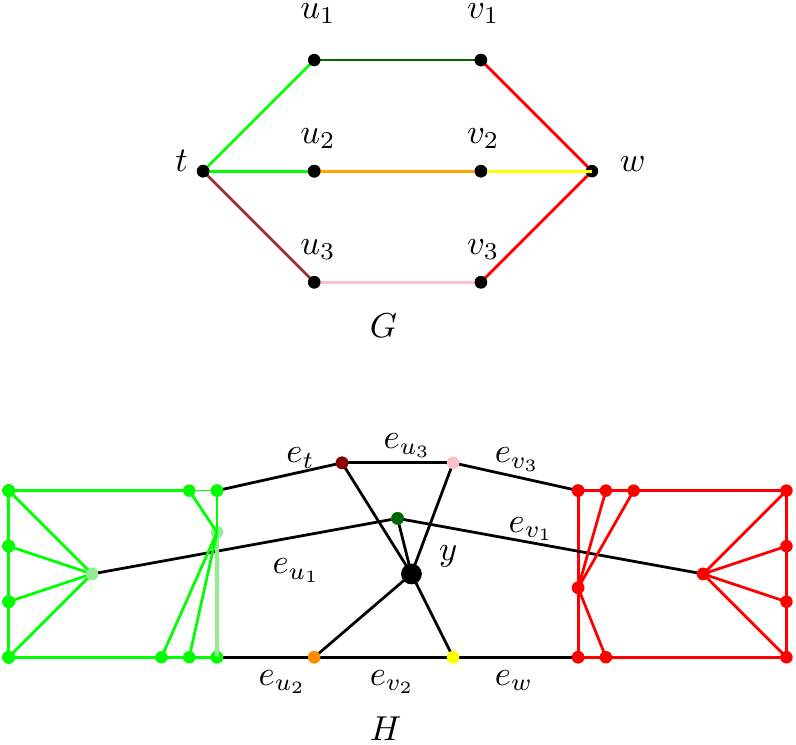}
        \caption{An example for the construction of $H$. The colors mark the legal decomposition $\mathcal{P}$.}\label{ajouteraussi}
\end{figure}

The following result shows that $H$ is indeed contained in the desired input domain.

\begin{lemma}\label{h-v}
    $H-a$ is 2-edge-connected for all $a \in V(H)$.
\end{lemma}
\begin{proof}
Suppose otherwise, so there is a vertex $a \in V(H)$ and a set $S \subseteq V(H)-a$ such that $d_{H-a}(S)\leq 1$.

First suppose that $a \neq y$. By symmetry, we may suppose that $y \in S$. First consider some $P_0 \in \mathcal{P}_1$ with $x_{P_0}\neq a$. If there is some $P_1 \in \mathcal{P}_1$ with $V(P_0)\cap V(P_1)\neq\emptyset$ and $x_{P_1}\neq a$, then $H-a$ contains the $x_{P_0}y$-paths $x_{P_0}y$ and $x_{P_0}x_{P_1}y$, so $x_{P_0}\in S$. Otherwise, there is some $P_2 \in \mathcal{P}_2$ with $V(P_0)\cap V(P_2)\neq\emptyset$ and $a \notin X_{P_2}$. Further, as $\mathcal{P}$ is legal, there is some $P_3 \in \mathcal{P}_1$ with $V(P_2)\cap V(P_3)\neq\emptyset$ and $x_{P_3}\neq a$. Hence $H-a$ contains the $x_{P_0}y$-path $x_{P_0}y$ and an $x_{P_0}y$-path passing through $X_{P_2}$ and $x_{P_3}$. This yields $x_{P_0}\in S$. We obtain $\bigcup_{P \in \mathcal{P}_1}x_P-a \subseteq S$.

Now consider some $P_0 \in \mathcal{P}_2$ with $a \notin X_{P_0}$. Observe that $H[X_{P_0}]$ is 2-edge-connected, hence we either have $X_{P_0}\subseteq S$ or $X_{P_0}\cap S= \emptyset$. As $\mathcal{P}$ is legal, there are paths $P_1,P_2 \in \mathcal{P}_1$ with $V(P_0)\cap V(P_1)\neq\emptyset, V(P_0)\cap V(P_2)\neq\emptyset$ and $a\notin \{x_{P_1},x_{P_2}\}$. Hence $H-a$ contains an edge from $X_{P_0}$ to $x_{P_i}$ for $i=1,2$, so $X_{P_0} \subseteq S$.

We obtain that $V(H)-S \subseteq X_P$ for some $P \in \mathcal{P}_2$ with $a \in X_P$. Let $H'$ be the graph that is obtained from $H$ by contracting $V(H)-X_{P}$ into a single vertex. It is easy to see that $H'-z$ is 2-edge-connected for all $z \in V(H')$. This yields $d_{H-a}(S)=d_{H'-a}(S)\geq 2$, a contradiction.

   Now suppose that $a=y$. Observe that for every $P \in \mathcal{P}_2$, we have that $H[X_{P}]$ is 2-edge-connected, hence we either have $X_{P}\subseteq S$ or $X_{P}\cap S= \emptyset$. Let $\mathcal{P}^S=\{P \in \mathcal{P}_1:x_P \in S\} \cup \{P \in \mathcal{P}_2:X_P \subseteq S\}$ and $\mathcal{P}^{-S}=\mathcal{P}-\mathcal{P}^S$. Further, let $V^S=\bigcup_{P \in \mathcal{P}^S}V(P)$ and $V^{-S}=\bigcup_{P \in \mathcal{P}^{-S}}V(P)$. Observe that $V^S\cup V^{-S}=V(G)$ and since $S,V(H)-S-y\neq \emptyset$ and by the construction, we have $\min\{|V^S|,| V^{-S}|\} \geq 2$. If one of $V^S-V^{-S}$ and $V^{-S}-V^S$ is empty, we have $|V^S \cap V^{-S}|=\min\{|V^S|,| V^{-S}|\} \geq 2$. Otherwise, observe that $G$ does not contain an edge linking $V^S-V^{-S}$ and $V^{-S}-V^S$, so $V^S \cap V^{-S}$ is a separator of $G$. As $G$ is 2-vertex-connected, again, we obtain $|V^S \cap V^{-S}| \geq 2$. Finally, observe that for every $w \in V^S \cap V^{-S}$, we have $e_w \in \delta_{H-y}(S)$. This yields $d_{H-y}(S)\geq |V^S \cap V^{-S}| \geq 2$, a contradiction. This finishes the proof.
\end{proof}

The following lemma is the key for the reduction in Section \ref{demo}. Its proof has some similarities with the one of Lemma \ref{h-v}, but we give it separately for the sake of readability.

\begin{lemma}\label{cuts}
$H$ is 3-edge-connected and the 3-edge-cuts of $H$ are the following:
\begin{enumerate}[(i)]
    \item $\delta_H(x_P)$ for every $P \in \mathcal{P}_1$,
    \item $\delta_H(X_P)$ for every $P \in \mathcal{P}_2$,
    \item $\delta_H(x_P^{i})$ for every $P \in \mathcal{P}_2$ and $i \in \{v,w,1\ldots,7\}$,
    \item $\delta_H(\{x_P^{u},x_P^2,\ldots,x_P^5\})$ for every $P \in \mathcal{P}_2$.
\end{enumerate}
\end{lemma}
\begin{proof}
It is easy to see that all the given cuts are 3-edge-cuts indeed. It hence suffices to prove that all cuts of size at most 3 are 3-edge-cuts of this form. Let $\emptyset \neq S \subsetneq V(H)$ with $d_H(S)\leq 3$. By symmetry, we may suppose that $y \in V(H)-S$. Observe that for every $P \in \mathcal{P}_2$, we have that $H[X_P]$ is 2-edge-connected. If there are distinct $P_1,P_2 \in \mathcal{P}_2$ such that $S \cap X_{P_i}$ and $S-X_{P_i}$ are nonempty for $i=1,2$, we obtain $d_H(S)\geq \sum_{i=1}^2d_{H[X_{P_i}]}(S \cap X_{P_i})\geq 2+2=4$, a contradiction. Hence there is at most one such path. For the rest of the proof, we distinguish the two cases whether this path exists or not.

\begin{case}
There exists a path $P^*\in \mathcal{P}_2$ such that $S \cap X_{P^*}$ and $S-X_{P^*}$ are nonempty.
\end{case}

\begin{claim}\label{zuerst}
For all $P \in \mathcal{P}_1$, we have $x_P \in V(H)-S$.
\end{claim}
\begin{proof}
Suppose otherwise. If there are two paths $P_1,P_2 \in \mathcal{P}_1$ with $x_{P_i}\in S$ for $i=1,2$, we obtain $d_H(S)\geq d_{H[X_{P^*}]}(S \cap X_{P^*})+\sum_{i=1}^2d_{H}(x_{P_i},y)\geq 2+1+1=4$, a contradiction. We may hence suppose that there is a unique path $P_1 \in \mathcal{P}_1$ with $x_{P_1}\in S$. As $G \in \mathcal{G}$ and $\mathcal{P}$ is legal, there is a vertex $v\in V(G)$ and a path $P_2 \in \mathcal{P}-\{P^*,P_1\}$ such that $v \in V(P_1 )\cap V(P_2)$. If the endvertex of $e_v$ which is distinct from $x_{P_1}$ is in $V(H)-S$, we obtain $d_H(S)\geq d_{H[X_{P^*}]}(S \cap X_{P^*})+d_{H}(x_{P_1},y)+|\{e_v\}|=2+1+1=4$, a contradiction. We obtain that this vertex is in $S$. By the choice of $P_1$, we obtain that $P_2 \in \mathcal{P}_2$. Further, by the choice of $P^*$, we obtain that $X_{P_2}\subseteq  S$. As $G \in \mathcal{G}$ and $\mathcal{P}$ is legal, we obtain that there is a path $P_3\in \mathcal{P}_1-P_1$ such that $V(P_2)\cap V(P_3)$ contains a vertex $w$. By the choice of $P_1$, we obtain $x_{P_3}\in V(H)-S$. This yields $d_H(S)\geq d_{H[X_{P^*}]}(S \cap X_{P^*})+d_{H}(x_{P_1},y)+d_H(x_{P_2}^w,x_{P_3})=2+1+1=4$, a contradiction.
\end{proof}

\begin{claim}\label{dann}
For all $P \in \mathcal{P}_2-P^*$, we have $X_P \subseteq V(H)-S$.
\end{claim}
\begin{proof}
Suppose otherwise, so by the choice of $P^*$, there is some $P_1 \in \mathcal{P}_2-P^*$ with $X_{P_1}\subseteq S$. As $G \in \mathcal{G}$ and $\mathcal{P}$ is legal, there are distinct paths $P_2,P_3,P_4 \in \mathcal{P}_1$ and $v_2,v_3,v_4 \in V(G)$ such that $v_i \in V(P_i)\cap V(P_1)$. By Claim \ref{zuerst}, we have $\{x_{P_2},x_{P_3},x_{P_4}\}\subseteq V(H)-S$. This yields that $d_H(S)\geq d_{H[X_{P^*}]}(S \cap X_{P^*})+\sum_{i=2}^4d_{H}(x_{P_1}^{v_i},x_{P_i})\geq 2+1+1+1=5$, a contradiction.
\end{proof}
By Claims \ref{zuerst} and \ref{dann}, we obtain that $S \subsetneq X_{P^*}$. By construction, we obtain that one of the cases $(iii)$ and $(iv)$ of Lemma \ref{cuts} occurs. This finishes the case.

\begin{case}
For all $P \in \mathcal{P}_2$, we have either $X_P\subseteq S$ or $X_P \cap S = \emptyset$.
\end{case}
\end{proof}

Let $\mathcal{P}_1^S$ be the sets of paths $P \in \mathcal{P}_1$ for which $x_P \in S$ holds and let $\mathcal{P}_2^S$ be the sets of paths $P \in \mathcal{P}_2$ for which $X_P \subseteq S$ holds. Next, let $\mathcal{P}_1^{-S}=\mathcal{P}_1-\mathcal{P}_1^S$ and $\mathcal{P}_2^{-S}=\mathcal{P}_2-\mathcal{P}_2^S$. Finally, let $V^S=\bigcup_{P \in \mathcal{P}_1^S \cup \mathcal{P}_2^S}V(P)$ and $V^{-S}=\bigcup_{P \in \mathcal{P}_1^{-S} \cup \mathcal{P}_2^{-S}}V(P)$. Observe that $V^S,V^{-S}\subseteq V(G)$.

\begin{claim}\label{cutgross}
$d_H(S)\geq |\mathcal{P}_1^S|+|V^S \cap V^{-S}|$.
\end{claim}
\begin{proof}
For every $P \in \mathcal{P}_1^S$, the edge $x_Py$ is contained in $\delta_H(S)$ and for every $v \in V^S \cap V^{-S}$, the edge $e_v$ is contained in $\delta_H(S)$. As all of these edges are distinct, the statement follows.
\end{proof}

We now distinguish several cases depending on the size of $\mathcal{P}_1^S$. 

\begin{subcase}
$\mathcal{P}_1^S=\emptyset$.
\end{subcase}

 As $G \in \mathcal{G}$, $\mathcal{P}$ is legal and $\mathcal{P}_1^S=\emptyset$, we obtain $\bigcup_{P \in \mathcal{P}_2^S}V(P) \subseteq V^S \cap V^{-S}$. If $\mathcal{P}_2^S$ contains at least two paths $P_1,P_2$, we obtain $|V^S \cap V^{-S}|\geq |V(P_1)|+\|V(P_2)| \geq 6$, a contradiction to Claim \ref{cutgross}. If $|\mathcal{P}_2^S|=1$, by construction, case $(ii)$ of Lemma \ref{cuts} occurs. If $|\mathcal{P}_2^S|=0$, we obtain $S=\emptyset$, a contradiction.

\begin{subcase}
$\mathcal{P}_1^S\neq\emptyset$.
\end{subcase}
By Proposition \ref{trivial} and $|V(G)|\geq 5$, we have $|\mathcal{P}_1|\geq 4$. Hence if $\mathcal{P}_1^{-S}=\emptyset$, we obtain by Proposition \ref{cutgross} that $d_H(S)\geq |\mathcal{P}_1^S|=|\mathcal{P}_1|-|\mathcal{P}_1^{-S}|\geq 4-0=4$, a contradiction, so $\mathcal{P}_1^{-S}\neq \emptyset$. As every path in $\mathcal{P}_1^{-S}$ contains two vertices, we obtain $| V^{-S}|\geq 2$. Similarly, the fact that $\mathcal{P}_1^S\neq\emptyset$ yields $| V^{S}|\geq 2$. If one of $V^S-V^{-S}$ and $V^{-S}-V^S$ is empty, we have $|V^S \cap V^{-S}|=\min\{|V^S|,| V^{-S}|\} \geq 2$.
Observe that every edge of $E(G)$ incident to a vertex in $V^S-V^{-S}$ is contained in a path of $\mathcal{P}^S$ and every edge of $E(G)$ incident to a vertex in $V^{-S}-V^{S}$ is contained in a path of $\mathcal{P}^{-S}$.
Hence, $G$ does not contain an edge linking $V^S-V^{-S}$ and $V^{-S}-V^S$, so $V^S \cap V^{-S}$ is a separator of $G$. As $G$ is 2-vertex-connected, again, we obtain $|V^S \cap V^{-S}| \geq 2$.
By Claim \ref{cutgross}, we obtain $3\geq d_H(S)\geq |\mathcal{P}_1^S|+|V^S \cap V^{-S}|\geq 1+2=3$. Hence equality holds throughout yielding that $\mathcal{P}_1^S$ contains a single path $P$ and $|V^S \cap V^{-S}|=2$. As $G \in \mathcal{G}$ and $\mathcal{P}$ is legal, we obtain $S=\{x_P\}$, so case $(i)$ of Lemma \ref{cuts} occurs.
\subsubsection{The main proof}\label{demo}
In this section, we show that our reduction works indeed. More formally, we prove the following statement.
\begin{lemma}
$(H,k+|V(G)|)$ is a positive instance of 34EDA if and only if $(G,k)$ is a positive instance of VC.
\end{lemma}
\begin{proof}
First suppose that $(G,k)$ is a positive instance of VC, so there is a vertex cover $S \subseteq V(G)$ with $|S|\leq k$. We now define a set $F \subseteq E(H)$. First, for all $v \in S$, we let $F$ contain the edge $e_v$. Further, for every $P=uvw \in \mathcal{P}_2$ with $v \in S$, we let $F$ contain the edges $x_P^{1}x_P^2,x_P^{3}x_P^4,x_P^{5}x_P^6$ and $x_P^{w}x_P^7$ and for every $P=uvw \in \mathcal{P}_2$ with $v \notin S$, we let $F$ contain the edges $x_P^{v}x_P^1,x_P^{2}x_P^3,x_P^{4}x_P^5$ and $x_P^{6}x_P^7$. This finishes the description of $F$. Observe that, by Proposition \ref{trivial}, we have $|F|=|S|+4|\mathcal{P}_2|\leq k+ |V(G)|$. In order to prove that the graph obtained from $H$ by doubling all edges in $F$ is 4-edge-connected, it suffices to prove that $F$ contains at least one edge of every 3-edge-cut of $H$ which are listed in Lemma \ref{cuts}.

First consider some $P=uv \in\mathcal{P}_1$. As $S$ is a vertex cover and $uv \in E(G)$, we obtain that $S$ contains one of $u$ and $v$. This yields that $F$ contains one of $e_u$ and $e_v$, hence at least one edge of $\delta_H(x_P)$.

Now consider some $P=uvw \in\mathcal{P}_2$. 

As $S$ is a vertex cover and $uv \in E(G)$, we obtain that $S$ contains one of $u$ and $v$. This yields that $F$ contains one of $e_u$ and $e_v$, hence at least one edge of $\delta_H(X_P)$.

Next observe that by construction $F$ contains an edge in $\delta_H(x_P^{i})$ for all $i \in \{1,\ldots,7\}$.

We now distinguish two cases.
\begin{case}
$v \in S$
\end{case}
By construction, we have $e_v \in F$, so $F$ contains an edge in $\delta_H(x_P^v)$. Next, by construction, we have $x_P^{w}x_P^7 \in F$, so $F$ contains an edge in $\delta_H(x_P^{w})$. Finally, $F$ contains the edge $x_P^{1}x_P^2$, so $F$ contains an edge in $\delta_H(\{x_P^{u},x_P^2,\ldots,x_P^5\})$.

\begin{case}
$v \notin S$
\end{case}
As $S$ is a vertex cover and $uv,vw \in E(G)$, we obtain that $u,w \in S$, so $e_u,e_w \in F$.
In particular, $F$ contains an edge in $\delta_H(x_P^w)$ and an edge in $\delta_H(\{x_P^{u},x_P^2,\ldots,x_P^5\})$. Finally, $F$ contains the edge $x_P^{v}x_P^1$, so $F$ contains an edge in $\delta_H(x_P^{v})$.
\medskip

Hence, by Lemma \ref{cuts}, we obtain that the graph obtained from $H$ by doubling the edges of $F$ is 4-edge-connected, so $(H,k+|V(G)|)$ is a positive instance of 34EDA.
\bigskip

Now suppose that $(H,k+|V(G)|)$ is a positive instance of 34EDA, so there is a set $F \subseteq E(H)$ with $|F|\leq k+|V(G)|$ such that the graph obtained from doubling every edge of $F$ is 4-edge-connected.

We say that a path $P=uv \in \mathcal{P}_1$ is {\bf nice} with respect to $F$ if $F$ contains at least one of $e_u$ and $e_v$ and that a path $P=uvw \in \mathcal{P}_2$ is {\it nice} with respect to $F$ if $F$ contains either $e_v$ or both $e_u$ and $e_w$ and $|F \cap E(H[X_P])|\geq 4$. We may suppose that $F$ is chosen among all feasible solutions to the instance of 34EDA so that the number of paths in $\mathcal{P}$ which are not nice with respect to $F$ is minimized.

\begin{claim}\label{1nice}
All paths in $\mathcal{P}_1$ are nice with respect to $F$.
\end{claim}
\begin{proof}
Suppose otherwise, so there is some $P=uv \in \mathcal{P}_1$ such that $P$ is not nice with respect to $F$. As $F$ contains an edge in $\delta_H(x_P)$, we obtain that $x_Py \in F$. Let $F'=F-x_Py\cup e_u$. Clearly, we have $|F'|=|F|\leq k+|V(G)|$. Further, it follows from Lemma \ref{cuts} that the graph obtained from $H$ by doubling all edges of $F'$ is 4-edge-connected. This contradicts the choice of $F$.
\end{proof}

\begin{claim}\label{2nice}
All paths in $\mathcal{P}_2$ are nice with respect to $F$.
\end{claim}
\begin{proof}
Let $P=uvw\in \mathcal{P}_2$.
We  first prove that $|F \cap E(H[X_P])|\geq 4$. Observe that $H[X_P]-\{x_P^{u},x_P^{v},x_P^{w}\}$ contains 7 vertices of degree 3 in $H$ and these vertices are only incident to edges of $E(H[X_P])$ in $H$. As the graph obtained from $H$ by doubling the edges of $F$ is 4-edge-connected, we obtain that $|F \cap E(H[X_P])|\geq 4$. Hence, if $e_v \in F$ or $\{e_u,e_w\}\subseteq F$, there is nothing to prove. Further, as $F$ contains an edge in $\delta_H(X_P)$, we obtain that $F$ contains at least one edge of $\{e_u,e_v,e_w\}$. Suppose for the sake of a contradiction that $F$ does not contain $e_v$ and contains exactly one of $e_u$ and $e_w$.
\begin{subclaim}\label{5}
$|F \cap E(H[X_P])|\geq 5$.
\end{subclaim}
\begin{proof}
First suppose that $F$ contains the edge $e_u$. Observe that $H[X_P]-\{x_P^{u}\}$ contains 9 vertices of degree 3 in $H$ and these vertices are only incident to edges of $E(H[X_P])$ in $H-\{e_v,e_w\}$. As $F$ does not contain $e_v$ and $e_w$ and the graph obtained from $H$ by doubling the edges of $F$ is 4-edge-connected, we obtain that $|F \cap E(H[X_P])|\geq 5$.

Now suppose that $H$ contains the edge $e_w$. Observe that $H[X_P]-\{x_P^{w}\}$ contains 8 vertices of degree 3 in $H$ and these vertices are only incident to edges of $E(H[X_P])$ in $H-\{e_u,e_v\}$. Further, the only set of 4 edges in $E(H[X_P])$ containing at least one edge of all corresponding 3-edge cuts is $\{x_P^{v}x_P^{1},x_P^{2}x_P^{3},x_P^{4}x_P^{5},x_P^{6}x_P^{7}\}$. However if $F \cap E(H[X_P])=\{x_P^{v}x_P^{1},x_P^{2}x_P^{3},x_P^{4}x_P^{5},x_P^{6}x_P^{7}\}$, then as $e_u \notin F$, $F$ does not contain any edge of the 3-edge cut $\delta_H(\{x_P^{u},x_P^2,\ldots,x_P^5\})$, a contradiction. We hence obtain $|F \cap E(H[X_P])|\geq 5$.
\end{proof}

Let $F'=F-E(H[X_P])\cup \{e_u,x_P^{v}x_P^{1},x_P^{2}x_P^{3},x_P^{4}x_P^{5},x_P^{6}x_P^{7}\}$. By Subclaim \ref{5}, we have $|F'|\leq|F|\leq k+|V(G)|$. Further, it follows from Lemma \ref{cuts} that the graph obtained from $H$ by doubling all edges of $F'$ is 4-edge-connected. This contradicts the minimality of $F$.
\end{proof}

We are now ready to define a vertex cover $S \subseteq V(G)$ of $G$. Namely, we let $S$ include a vertex $v$ if $e_v \in F$. Observe that by Claim \ref{2nice} and Proposition \ref{trivial}, we have $|S|\leq|F|-4|\mathcal{P}_2|\leq (k+|V(G)|)-|V(G)|=k$. Now consider some $uv \in E(G)$. As $\mathcal{P}$ is legal, we obtain that $uv \in E(P)$ for some $P \in \mathcal{P}$. As $P$ is nice with respect to $F$ and by definition of $S$, we obtain that $S$ contains at least one of $u$ and $v$. Hence $S$ is a vertex cover of $G$. This finishes the proof.
\end{proof}

We wish to remark that the same proof technique can be used to prove that the problem of minimizing the arcs that are doubled is APX-hard, relying on a corresponding result for Cubic Vertex Cover by Alimonti and Kann \cite{AK}. Similar results follow for the partial orientation problems.
\subsection{Approximation algorithms}\label{approx34}

In this section, we give another application of the connections established in Lemmas \ref{lem:one} and \ref{lem:two}. Namely, we give an approximation algorithm for the problem of making a graph 4-edge-connected by doubling a minimum number of edges and conclude the existence of approximation algorithms for partial orientation problems from this.
We consider the following minimization problem which is a natural generalization of 4EDA:

\medskip
\begin{center}
\begin{tabular}{|ccc|}\hline
 & \begin{minipage}{14cm}
\vspace{2mm}
\noindent {\bf Minimum 4-Edge Doubling Augmentation \textbf{M4EDA}}
\smallskip

\noindent\textbf{Input:} A 2-edge-connected graph $G$.
\smallskip

\noindent\textbf{Question:} What is the minimum cardinality $OPT(G)$ of a set $F$ of edges in $E(G)$ such that the graph obtained from $G$ by doubling the edges in $F$ is 4-edge-connected?
\medskip
\vspace{2mm}
\end{minipage} & \\ \hline
\end{tabular}
\end{center}

Observe that the condition that the input graph is 2-edge-connected is necessary for a feasible solution to exist. In order to obtain an approximation result for M4EDA, we heavily rely on some previous work of Cecchetto, Traub and Zenklusen \cite{ctz}. They consider an optimization problem which contains the following problem as a special case:

\begin{center}
\begin{tabular}{|ccc|}\hline
 & \begin{minipage}{14cm}
\vspace{2mm}
\noindent {\bf Restricted 3-4 Edge-Connectivity Augmentation \textbf{R34ECA}}
\smallskip

\noindent\textbf{Input:} A 3-edge-connected graph $G$ and a graph $H$ with $V(H)=V(G)$ such that $(V(G),E(G)\cup E(H))$ is 4-edge-connected.
\smallskip

\noindent\textbf{Question:} What is the minimum cardinality $OPT'(G,H)$ of a set $F$ of edges in $E(H)$ such that the graph obtained from $G$ by adding the edges in $F$ is 4-edge-connected?
\medskip
\vspace{2mm}
\end{minipage} & \\ \hline
\end{tabular}
\end{center}

The following result follows from a result in \cite{ctz}.

\begin{theorem}\label{zenk}
    There is an algorithm $A_0$ whose input is an instance $(G,H)$ of R34ECA and that outputs a set $F\subseteq E(H)$ such that the graph obtained from $G$ by adding the edges of $F$ is 4-edge-connected and $|F|\leq \alpha OPT'(G,H)$ holds where $\alpha=1.393\ldots$ is a constant.
\end{theorem}

Using this, we obtain the following result:

\begin{theorem}\label{appro}
    There is an algorithm $A$ whose input is an instance $G$ of M4EDA and that outputs a set $F\subseteq E(G)$ such that the graph obtained from $G$ by doubling the edges of $F$ is 4-edge-connected and $|F|\leq \alpha OPT(G,H)$ holds where $\alpha=1.393\ldots$ is the same constant as in Theorem \ref{zenk}.
\end{theorem}
\begin{proof}
    We first outline the algorithm. Let $G$ be an instance of M4EDA and let $F_1$ be the set if edges in $E(G)$ that are contained in a 2-edge-cut of $G$. Let $G'$ be the graph obtained from $G$ by doubling the edges in $F_1$. We further let $H$ be the graph with $V(H)=V(G)$ that contains a copy $e'$ of every $e \in E(G)-F_1$. Observe that $(V(G),E(G')\cup E(H))$ can be obtained from $G$ by doubling all edges and is hence 4-edge-connected. We may hence apply the algorithm $A_0$ to the instance $(G',H)$ of R34ECA. Let $F_2'$ be the set of edges returned by $A_0$ and let $F_2$ be the set of corresponding edges in $E(G)$. We now let $A$ return $F=F_1 \cup F_2$.

    Using the fact that $A_0$ is polynomial, it is easy to see that $A$ is also polynomial. Further, as $F_2'$ is a feasible solution for the instance $(G,H)$ of R34ECA, we obtain that the graph obtained from $G'$ by doubling the edges in $F_2$ is 4-edge-connected and hence $F$ is a feasible solution for the instance $G$ of M4EDA. Now let $F^*$ be an optimal solution for the instance $G$ of M4EDA. As every edge in $F_1$ is contained in a 2-edge-cut of $G$, we obtain that $F_1 \subseteq F^*$. Let $F''$ be the set of edges in $E(H)$ that correspond to the edges in $F^*-F_1$. As the graph obtained from $G'$ by doubling the edges in $F^*-F_1$ is 4-edge-connected, we obtain $F''$ is a feasible solution for the instance $(G,H)$ of R34ECA. This yields $|F|=|F_1|+|F_2|\leq |F_1|+\alpha OPT'(G,H)\leq |F_1|+\alpha |F''|=|F_1|+\alpha (|F^*|-|F_1|)\leq \alpha |F^*|=\alpha OPT(G)$. Hence $A$ has all the desired properties. This finishes the proof.
\end{proof}

Combining Theorem \ref{appro} with Lemmas \ref{lem:one} and \ref{lem:two}, respectively, we obtain the following conclusions for partial orientations. While the first one is an immediate application of Theorem \ref{appro}, for the second one an argument similar to the one in the proof of Theorem \ref{appro} can be used. We leave the details to the interested reader.

\begin{corollary}
There is an algorithm whose input is a 2-edge-connected graph $G$ and that computes a 2-arc-strong partial orientation $M$ of $G$ such that the number of undirected edges in $M$ is at most $\alpha$ times bigger than in any other 2-arc-strong partial orientation of $G$.
\end{corollary}

\begin{corollary}
There is an algorithm whose input is a 2-vertex-connected graph $G$ and that computes a 2-strong partial orientation $M$ of $G$ such that the number of undirected edges in $M$ is at most $\alpha$ times bigger than in any other 2-strong partial orientation of $G$.
\end{corollary}

\subsection{2 to 3-edge-connectivity augmentation by doubling edges is polynomial}\label{sec23}

The NP-hardness of 34EDA raises the question whether this problem becomes better tractable when weaker connectivity conditions are seeked for. For making a connected graph 2-edge-connected, this is easily seen to be the case as the set of all bridges forms an optimal solution. In this section, we show that a positive algorithmic result is also available for 3-edge-connectivity, even in the more general weighted setting.
We consider the following problem:

\medskip
\begin{center}
\begin{tabular}{|ccc|}\hline
 & \begin{minipage}{14cm}
\vspace{2mm}
\noindent {\bf Weighted 2-3 Edge Doubling Augmentation \textbf{W23EDA}}
\smallskip

\noindent\textbf{Input:} A 2-edge-connected graph $G$, a weight function $w:E(G)\rightarrow \mathbb{R}_{\geq 0}$, an integer $k$.
\smallskip

\noindent\textbf{Question:} Can $H$ be made 3-edge-connected by doubling at set $F$ of edges with $w(F)<k$?
\medskip
\vspace{2mm}
\end{minipage} & \\ \hline
\end{tabular}
\end{center}

\smallskip
We prove the following:

\begin{theorem}\label{23easy}
W23EDA can be solved in polynomial time. Moreover, for positive instances, an optimal solution can be found in polynomial time.
\end{theorem}

Let $G$ be a 2-edge-connected graph and $u,v \in V(G)$. We say that $u \sim v$ if $\lambda_G(u,v)\geq 3$. Observe that $\sim$ is an equivalence relation on $V(G)$. We let $Q_G$ be the graph that contains a vertex $v$ for each eqiuvalence class $B_v$ of $\sim$ and that contains an edge $uv$ for every edge of $G$ linking the eqiuvalence classes $B_u$ and $B_v$. A {\bf cactus} is an undirected graph $G$ that satisfies $\lambda_G(u,v)=2$ for all distinct $u,v \in V(G)$. 
\begin{lemma}\label{gqcactus}
    Let $G$ be a 2-edge-connected graph. Then $Q_G$ is a cactus.
\end{lemma}
\begin{proof}
    Let $u,v \in V(Q_G)$ with $u \neq v$ and consider some $X \subseteq V(Q_G)$ with $u \in X$ and $v \in V(Q_G)-X$. Let $\bar{X}=\bigcup_{x \in X}B_x$. Then $d_{Q_G}(X)=d_G(\bar{X})\geq 2$, so $\lambda_{Q_G}(u,v)\geq 2$.

    Now consider some $u'\in B_u$ and $v'\in B_v \neq B_u$. As $\lambda_G(u',v')=2$, there is some $\bar{X}\subseteq V(G)$ with $u' \in \bar{X}$, $v' \in V(G)-\bar{X}$ and $d_G(\bar{X})=2$. By definition of $Q_G$, we have $B_w\cap \bar{X}=\emptyset$ or $B_w \subseteq \bar{X}$ for all $w \in V(Q_G)$. Let $X=\{w \in V(Q_G):B_w \subseteq \bar{X}\}$. Observe that $u \in X$ and $v \in V(Q_G)-X$. This yields $\lambda_{Q_G}(u,v)\leq d_{Q_G}(X)=d_G(\bar{X})=2$.

    Hence $\lambda_{Q_G}(u,v)=2$ and so the statement follows.
\end{proof}

Let $G$ be a 2-edge-connected graph and $\bar{w}:E(G)\rightarrow \mathbb{R}_{\geq 0}$ a weight function. Then $w:E(Q_G)\rightarrow \mathbb{R}_{\geq 0}$ denotes the weight function in which $w(e)=\bar{w}(\bar{e})$ holds for every $e \in E(Q_G)$ where $\bar{e}$ is the edge in $E(G)$ that corresponds to $e$.

\begin{lemma}\label{gqgleich}
  Let $G$ be a 2-edge-connected graph, $\bar{w}:E(G)\rightarrow \mathbb{R}_{\geq 0}$ a weight function and $k$ a constant. Then $(G,\bar{w},k)$ is a positive instance of $W23EDA$ if and only if $(Q_G,w,k)$ is a positive instance of $W23EDA$. Further, an optimal solution for $(G,\bar{w})$ can be obtained from an optimal solution of $(Q_G,w)$.
\end{lemma}
\begin{proof}
    First suppose that $(Q_G,w,k)$ is a positive instance of $W23EDA$, so there is an edge set $F\subseteq E(Q_G)$ with $w(F)\leq k$ such that the graph $Q_G'$ obtained from $Q_G$ by doubling all the edges of $F$ is 3-edge-connected. Let $\bar{F} \subseteq E(G)$ be the set of edges corresponding to $F$ and let $G'$ be the graph obtained from $G$ by doubling the edges of $\bar{F}$. Clearly, $\bar{F}$ can be constructed from $F$ in polynomial time and  we have $\bar{w}(\bar{F})=w(F)$. Now consider some $\bar{X}\subseteq V(G)$. If there is some $v \in V(Q_G)$ such that $X\cap B_v \neq \emptyset$ and $X-B_v \neq \emptyset$, we obtain $d_{G'}(\bar{X})\geq d_G(\bar{X})\geq 3$ by the definition of $Q_G$. Otherwise, let $X=\{w \in V(Q_G):B_w \subseteq \bar{X}\}$. We obtain $d_{G'}(\bar{X})= d_{Q_G'}(X)\geq 3$. Hence $G'$ is 3-edge-connected and $(G,\bar{w},k)$ is a positive instance of $W23EDA$.

    Now suppose that $(G,\bar{w},k)$ is a positive instance of $W23EDA$, so there is an edge set $\bar{F}\subseteq E(G)$ with $\bar{w}(\bar{F})\leq k$ such that the graph $G'$ obtained from doubling the edges of $\bar{F}$ is 3-edge-connected. Let $F\subseteq E(Q_G)$ be the set of edges corresponding to $F$ and let $Q_G'$ be the graph obtained from $Q_G$ by doubling the edges of $F$. Clearly, we have $w(F)=\bar{w}(\bar{F})$. For some $X \subseteq V(Q_G)$, let $\bar{X}=\bigcup_{x \in X}B_x$. We obtain $d_{Q_G'}(X)=d_{G'}(\bar{X})\geq 3$. Hence $Q_G'$ is 3-edge-connected and $(Q_G,w,k)$ is a positive instance of $W23EDA$.

\end{proof} 
For the next lemma, we need the following well-known property of cactuses.

\begin{proposition}\label{trivialC}
   Every cactus contains a vertex of degree 2.
\end{proposition}

\begin{lemma}\label{baum}
    Let $G$ be a cactus and $F \subseteq E(G)$. Then the graph $G'$ obtained from $G$ by doubling the edges of $G$ is 3-edge-connected if and only if $(V(G),F)$ is connected.
\end{lemma}
\begin{proof}
    First suppose that $(V(G),F)$ is connected and let $X \subseteq V(G)$. Then $d_{G'}(X)\geq d_G(X)+d_F(X)\geq 2+1=3$, so $G'$ is 3-edge-connected.

    Now suppose that $G'$ is 3-edge-connected and for the sake of a contradiction that $(V(G),F)$ is not connected. Further suppose that the size of $G$ is minimal among all graphs admitting an edge set with that property. By Proposition \ref{trivialC}, there is some $v \in V(G)$ with $d_G(v)=2$. Clearly, $F$ contains an edge $e \in \delta_G(v)$. Observe that $G/e$ is a cactus and the graph obtained from $G/e$ by doubling the edges in $F-e$ is 3-edge-connected. Further $(V(G/e),F-e)$ is not connected, a contradiction to the minimality of $G$.
\end{proof}
\begin{proof}(of Theorem \ref{23easy})

By Lemmas \ref{gqcactus} and \ref{gqgleich}, it suffices to prove the statement for cactuses. By Lemma \ref{baum}, this can be done by finding a minimum spanning tree of the cactus with respect to the given weight function. This can be done in polynomial time, for example using the algorithm of Kruskal, see \cite{kv}.
\end{proof}

\section{Deorientations}\label{sec:deor}

In a digraph $D$, the operation of replacing an arc $a \in A(D)$ by an undirected edge linking the same two vertices is called {\bf deorienting} the arc.  
Given a digraph $D$ and an integer $k$, we wish to know whether we can deorient at most $k$ of the arcs of $D$ so that the obtained mixed graph satisfies certain connectivity properties. Let $deor_k(D)$, respectively $deor_k^{arc}(D)$ denote the minimum number of arcs one needs to deorient in $D$ to obtain a mixed graph which is $k$-strong, respectively $k$-arc-strong. Clearly $deor_1(D)=deor_1^{arc}(D)$.
The following result, which shows that $deor_1(D)$ can be found in polynomial time, is a consequence of the theorem of Lucchesi and Younger \cite{LY}, see also \cite[Section 13.1]{BG}.
\begin{theorem}\label{strong}
Let $D=(V,A)$ be a digraph and $k$ a positive integer. Then we can decide in polynomial time whether there exists a strongly connected deorientation of $D$ with at most $k$ edges.
\end{theorem}

In \cite{BG} (Problem 14.6.6), the first author and Gutin raised the question whether Theorem \ref{strong} can be generalized for stronger connectivity properties. The main result of this section is that there is no hope to do so as soon as the resulting mixed graph is required to be $k$-strong for some $k \geq 3$. More concretely, we show that computing the minimum number of arcs we need to deorient in a given digraph to obtain a 3-strong digraph is NP-hard. The following two problems are left for further research. The second one is already mentioned in \cite{bangDAM136}.

\begin{problem}
Determine the complexity of deciding for a given digraph $D$ and an integer $k$; whether $D$ has a deorientation with at most $k$ edges  that is $2$-strong.
\end{problem}

\begin{problem}\label{probarc}
For some integer $\ell \geq 2$, determine the complexity of deciding for a given digraph $D$ and an integer $k$; whether $D$ has a deorientation with at most $k$ edges that is $\ell$-arc-connected.
\end{problem}

The rest of this section is structured as follows: In Section \ref{3somdure}, we prove our main hardness result mentioned above. The remaining parts contain some smaller results on deorientation problems. In Section \ref{degrees}, we show that we can find in polynomial time a minimum set of arcs whose  deorientation makes the arising mixed graph satisfy certain degree conditions. In Section \ref{local}, we show that the problem of finding a minimum number of arcs whose deorientation makes a given digraph satisfy some local arc-connectivity requirements is NP-hard. In Section \ref{appro2}, we give an approximation algorithm for making a given digraph satisfy a global arc-connectivity requirement.

\subsection{Deorienting to get a 3-strong mixed graph is NP-hard}\label{3somdure}

This section is concerned with proving that the problem of deciding whether a 3-strong mixed graph can be obtained from a given digraph by deorienting a given number of arcs is NP-hard. We need some preliminaries for our reduction:
For a mixed graph $M$ and some $S \subseteq V(M)$, we say that $M$ is {\bf $\mathbf{k}$-strong in $S$} if there are $k$ vertex-disjoint paths from $s_1$ to $s_2$ in $M$ for any $s_1,s_2 \in S$.\\

The following is easy to prove (see Exercise 14.8 in \cite{BG} for a related result).

\begin{proposition}\label{connin}
Let $M$ be a mixed graph which is $k$-strong in $S$ and let $v \in V(M)-S$. Further, suppose that there are $k$ $vs_i$-paths $P_1,\ldots,P_k$ in $M$ with $s_i \in S$ for $i=1,\ldots,k$ and $V(P_i)\cap V(P_j)=v$ for all $i,j \in \{1,\ldots,k\}$ with $i \neq j$ and that there are $k$ $s'_iv$-paths $P'_1,\ldots,P'_k$ in $M$ with $s'_i \in S$ for $i=1,\ldots,k$ and $V(P'_i)\cap V(P'_j)=v$ for all $i,j \in \{1,\ldots,k\}$ with $i \neq j$. Then $M$ is $k$-strong in $S \cup v$.
\end{proposition}

For our reduction, we need the following problem:
\medskip
\begin{center}
\begin{tabular}{|ccc|}\hline
 & \begin{minipage}{14cm}
\vspace{2mm}

\textbf{3-Bounded MAX 2-SAT (3BMAX2SAT)}
\medskip

\textbf{Input:} A set of variables $X$, a set of clauses $\mathcal{C}$ each containing exactly two literals such that every variable of $X$ appears exactly 3 times in $\mathcal{C}$, at least once in positive and at least once in negated form, and an integer $\ell$.
\medskip

\textbf{Question:} Is there an assignment $\Phi:X \rightarrow \{TRUE,FALSE\}$ such that at least $\ell$ clauses of $\mathcal{C}$ are satisfied?

\medskip

\vspace{2mm}
\end{minipage} & \\ \hline
\end{tabular}
\end{center}

\smallskip
We use the following result that is implicitely proven by Berman and Karpinski in \cite{bk}.

\begin{proposition}\label{thom}
3BMAX2SAT is NP-hard.
\end{proposition}
In order to simplify our reduction, we need the following slight adaption of this problem:
\medskip
\begin{center}
\begin{tabular}{|ccc|}\hline
 & \begin{minipage}{14cm}
\vspace{2mm}

\textbf{Special 3-Bounded MAX 2-SAT (S3BMAX2SAT)}
\medskip

\textbf{Input:} A set of variables $X$, a set of clauses $\mathcal{C}$ each containing exactly two literals such that every variable appears exactly twice in positive and exactly once in negated form in $\mathcal{C}$ and an integer $\ell$.
\medskip

\textbf{Question:} Is there an assignment $\Phi:X \rightarrow \{TRUE,FALSE\}$ such that at least $\ell$ clauses of $\mathcal{C}$ are satisfied?

\medskip

\vspace{2mm}
\end{minipage} & \\ \hline
\end{tabular}
\end{center}

\begin{proposition}\label{thom}
S3BMAX2SAT is NP-hard.
\end{proposition}
\begin{proof}
    Let $(X,\mathcal{C},\ell)$ be an instance 3BMAX2SAT. Let $X^-$ be the set of variables in $X$ that appear once in positive and twice in negated form. Let $\mathcal{C}'$ be the set of clauses which is obtained from $\mathcal{C}$ by negating the literals associated to the variables in $X^-$. Then $(X,\mathcal{C}',\ell)$ is an instance of S3BMAX2SAT and it is easy to see that $(X,\mathcal{C}',\ell)$ is a positive instance of S3BMAX2SAT if and only if $(X,\mathcal{C},\ell)$ is a positive instance of 3BMAX2SAT.
\end{proof}
Formally, we consider the following problem:
\medskip

\begin{center}
\begin{tabular}{|ccc|}\hline
 & \begin{minipage}{14cm}
\vspace{2mm}
\textbf{3-Strong DeOrientation (3SDO)}
\medskip

\textbf{Input:} A digraph $D$, an integer $k$.
\medskip

\textbf{Question:} Can $D$ be made 3-strong by deorienting at most $k$ arcs?

\vspace{2mm}
\end{minipage} & \\ \hline
\end{tabular}
\end{center}
\medskip
The following is our main result on deorientations.
\begin{theorem}\label{3sdohard}
For every $\ell \geq 3$, it is NP-hard to decide whether a given digraph has a deorientation with at most $k$ edges that is $\ell$-strong where $k$ is part of the input.
\end{theorem}

\begin{proof}
Observe that for a given integer $\ell \geq 4$, an instance $(D,k)$ of 3SDO is positive if and only if the graph obtained from adding $\ell-3$ new vertices and linking them by digons to all vertices of $D$ and to each other can be made $\ell$-strong by deorienting at most $k$ arcs. Hence, it suffices to prove the statement for 3SDO. We show this by describing a polynomial reduction from S3BMAX2SAT. 

Let $(X,\mathcal{C},\ell)$ be an instance of S3BMAX2SAT. We now create an instance $(D=(V,A),k)$ of 3SDO. For $x \in X$ and $C \in \mathcal{C}$, we say that $(x,C)$ is an {\bf incident pair} if $x \in C$ or $\bar{x}\in C$ and we let $\Gamma$ denote the set of incident pairs.

For every incident pair $(x,C)$, we let $D$ contain 3 vertices $p_{(x,C)},q_{(x,C)}$, and $s_{(x,C)}$ and a digon linking every pair of these vertices.

Now, for every $x \in X$, we let $D$ contain 10 more vertices $p_x,q_x,s^1_x,s^2_x,s_x^3,s_x^4,w^1_x,w^2_x,w^3_x$, and $w^4_x$. We add a digon between any pair of vertices in $\{p_x,q_x,s^3_x\}$ . Next, for every $C \in \mathcal{C}$, we let $D$ contain two more vertices $v_C$ and $s_C$ and we add an arc from $s_C$ to $v_C$.


Now, for every $x \in X$, let $C_1,\ldots,C_3$ be an ordering of the clauses containing $x$ or $\bar{x}$ such that $x \in C_1,C_3$ and $\bar{x}\in C_2$.\\ Let $V_x=\{p_{(x,C_1)},q_{(x,C_1)},p_{(x,C_2)},q_{(x,C_2)},p_{(x,C_3)},q_{(x,C_3)},p_x,q_x,w^1_x,w^2_x,w^3_x,w^4_x\}$. We add a digon between $w^{i}_x$ and $s^{j}_x$ for $i=1,\ldots,4$ and $j=1,2$.\\ Further, we add the following arcs: $q_{(x,C_1)}w_x^1,p_{(x,C_2)}w_x^1, q_{(x,C_2)}w_x^2,p_{(x,C_3)}w_x^2,q_{(x,C_3)}w_x^3,p_{x}w_x^3,q_{x}w_x^4,p_{(x,C_1)}w_x^4,\\ p_{(x,C_1)}v_{C_1}, v_{C_1}q_{(x,C_1)},q_{(x,C_2)}v_{C_2}, v_{C_2}p_{(x,C_2)},p_{(x,C_3)}v_{C_3}, v_{C_3}q_{(x,C_3)},q_{x}s_x^4, s_x^4p_{x}$.

Further, we set $S=\bigcup_{x \in X}\{s_x^1,\ldots,s_x^4\}\cup \bigcup_{(x,C)\in \Gamma}s_{(x,C)}\cup \bigcup_{C \in \mathcal{C}}s_C$ and we let $D$ contain a digon between any pair of vertices of $S$. Finally, we set $k=6|X|+|\mathcal{C}|-\ell$. An illustration for the part of the graph associated to some fixed $x \in X$ can be found in Figure \ref{deor1}.
\begin{figure}[H]
    \centering
        \includegraphics[width=.6\textwidth]{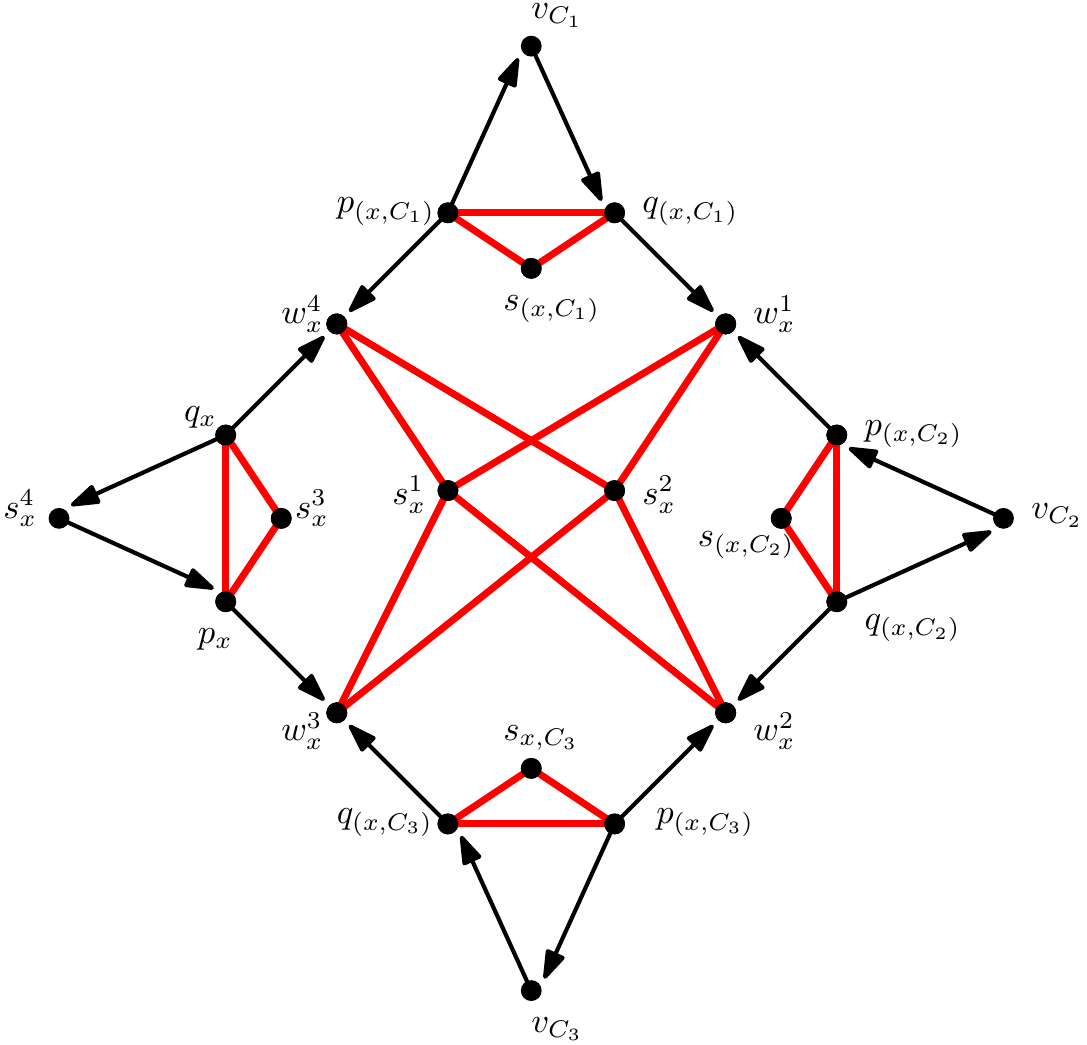}
        \caption{An example for $D[V_x\cup \{s_1^x,\ldots,s_4^x,v_{C_1},\ldots,v_{C_3}\}]$ for some $x \in X$. The thick red edges indicate digons. The digons linking pairs of vertices in $S$ have been omitted due to space restrictions.}\label{deor1}
\end{figure}
This finishes the description of $(D=(V,A),k)$. An illustration for a small instance of S3BMAX2SAT can be found in Figure \ref{deor4}.

\begin{figure}[H]
    \centering
        \includegraphics[width=\textwidth]{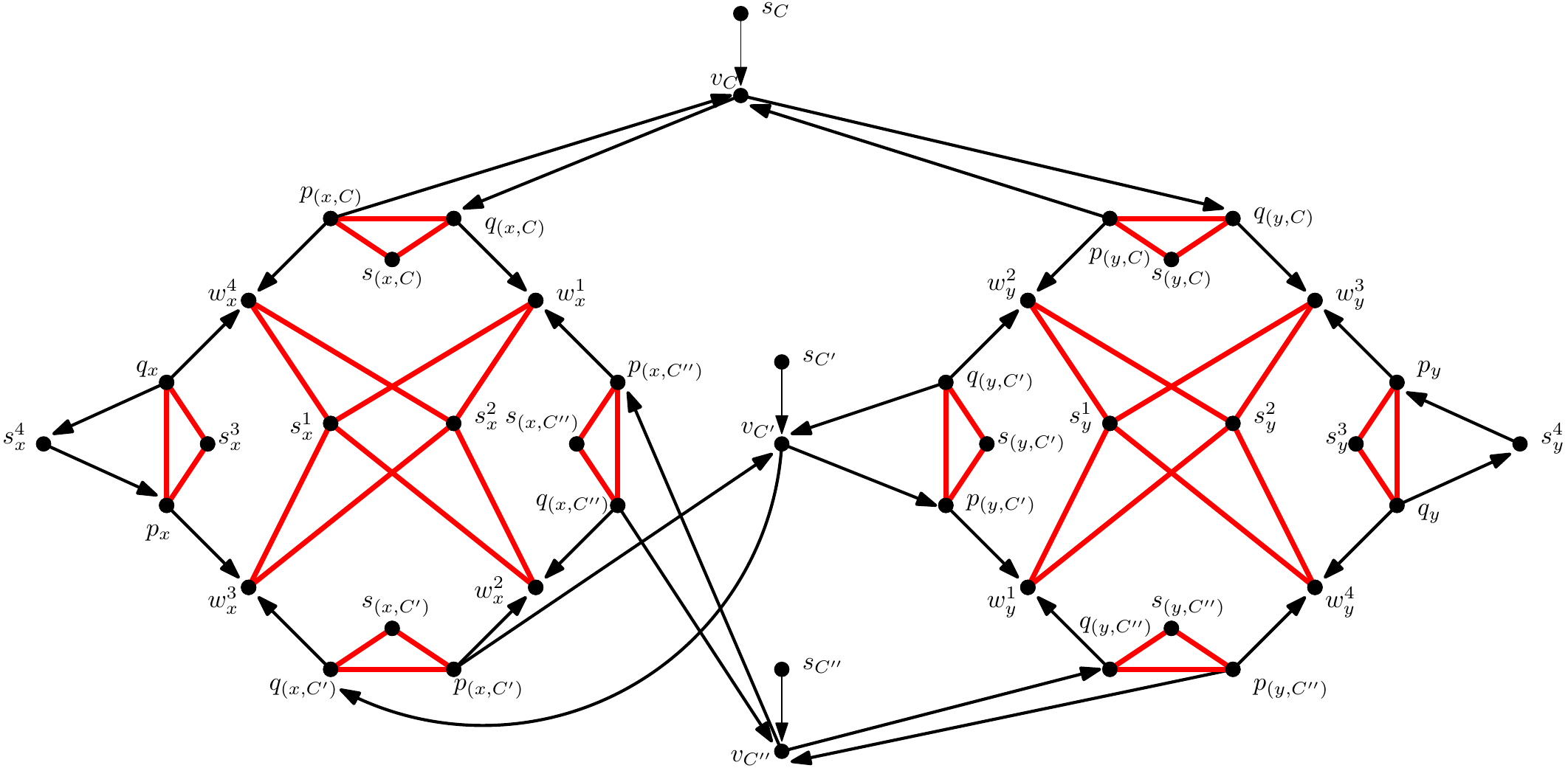}
        \caption{An example for the construction of $D$ where $X=\{x,y\}$ and $\mathcal{C}=\{C=\{x,y\},C'=\{x,\bar{y}\},C''=\{\bar{x},y\}\}$. In the construction the ordering $C,C'',C'$ of the clauses containing $x$ or $\bar{x}$ and the ordering $C'',C',C$ of the clauses containing $y$ or $\bar{y}$ are used. Again, the thick red edges indicate digons and the digons linking pairs of vertices in $S$ have been omitted due to space restrictions.}\label{deor4}
\end{figure}

We show in the following that $(D,k)$ is a positive instance of 3SDO if and only if $(X,\mathcal{C},\ell)$ is a positive instance of S3BMAX2SAT.

First suppose that $(X,\mathcal{C},\ell)$ is a positive instance of S3BMAX2SAT, so there is an  assignment $\phi:X\rightarrow\{TRUE,FALSE\}$ that satisfies at least $\ell$ clauses of $\mathcal{C}$. Let $\mathcal{C}'$ be the set of clauses in $\mathcal{C}$ which are satisfied by $\phi$ and $\mathcal{C}''=\mathcal{C}-\mathcal{C}'$. Let $F_1 \subseteq A$ be the set containing the following arcs:
\begin{itemize} 
\item the arcs $q_{(x,C_1)}w_x^1,q_{(x,C_2)}w_x^2,q_{(x,C_3)}w_x^3,q_{x}w_x^4,p_{(x,C_1)}v_{C_1}$ and $p_{(x,C_3)}v_{C_3}$ for all $x \in X$ with $\phi(x)=TRUE$ and 
\item the arcs $p_{(x,C_1)}w_x^4,p_{(x,C_2)}w_x^1,p_{(x,C_3)}w_x^2,p_xw_x^3,q_{(x,C_2)}v_{C_2}$ and $q_xs_x^4$ for all $x \in X$ with $\phi(x)=FALSE$.
\end{itemize}
An illustration can be found in Figure \ref{deor2}.
\begin{figure}[H]
  \includegraphics[width=.53 \linewidth]{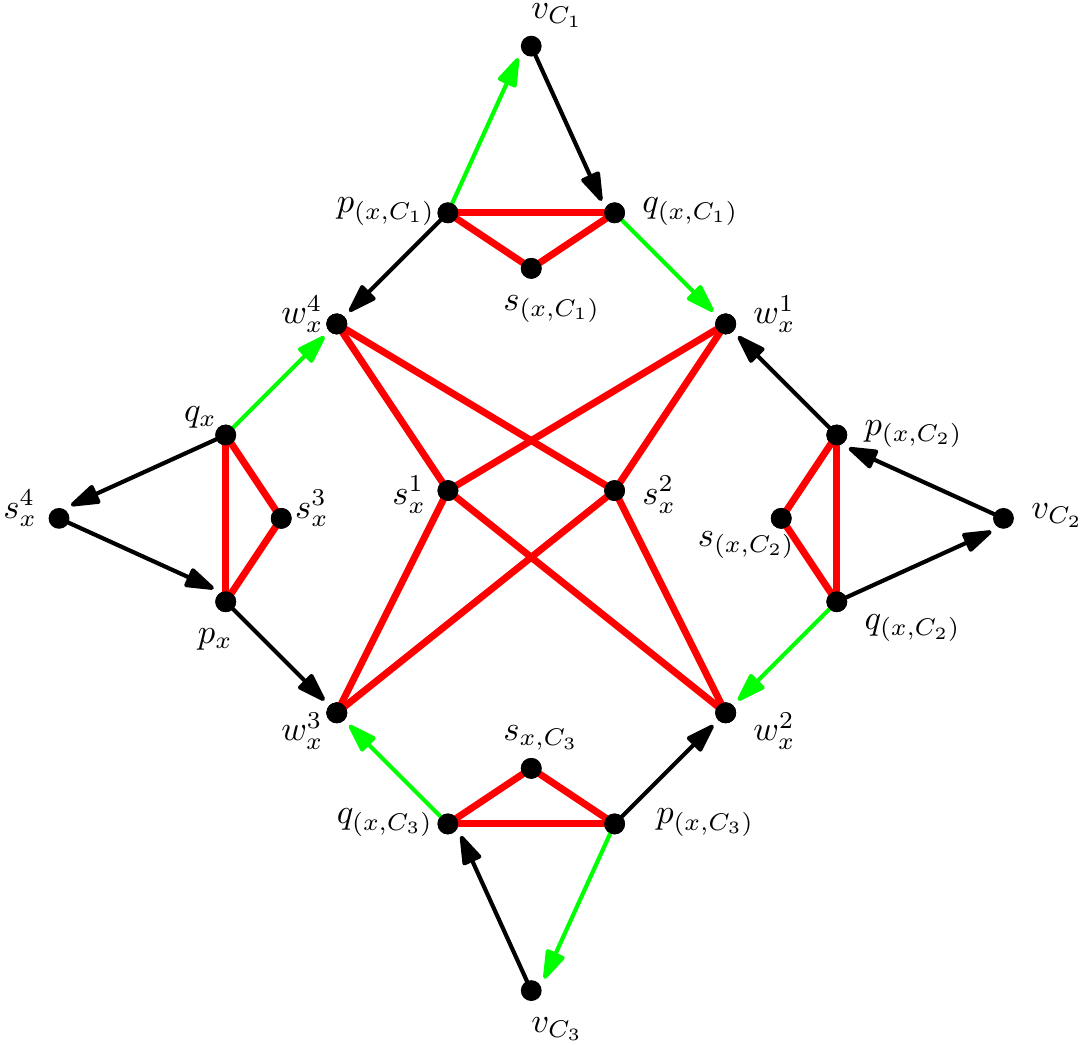}  \hskip 0.1truecm \includegraphics[width=.53 \linewidth]{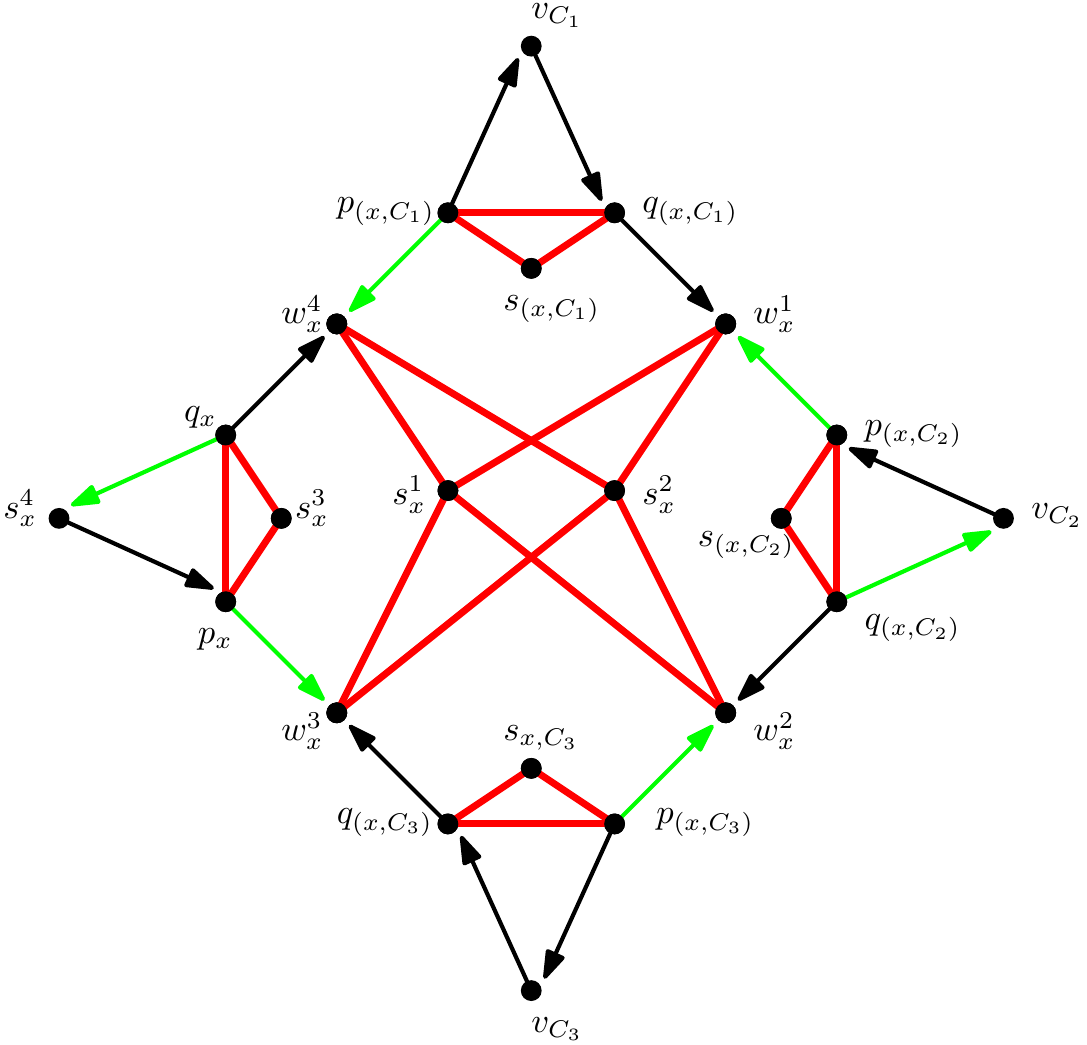}
  \caption{An illustration of the definition of $F_1$. The arcs contained in $F_1$ if $\phi(x)=TRUE$ are depicted in green in the left drawing and the arcs contained in $F_1$ if $\phi(x)=FALSE$ are depicted in green in the right drawing. Again, the thick red edges indicate digons and the digons linking pairs of vertices in $S$ have been omitted due to space restrictions.}\label{deor2}
\end{figure}Observe that $| F_1|=6|X|$ and that for some $C \in \mathcal{C}$, we have that $F_1$ contains an arc entering $v_C$ if and only if $C \in \mathcal{C'}$. Now we create a set $F_2\subseteq A$ that contains the arc $s_Cv_C$ for every $C \in \mathcal{C}''$. Let $F=F_1 \cup F_2$ and observe that $|F|=|F_1|+|F_2|=6|X|+|\mathcal{C}''|=6|X|+|\mathcal{C}|-|\mathcal{C}'|\leq 6|X|+|\mathcal{C}|-\ell=k$. Let $M$ be the mixed graph obtained from $D$ by deorienting the arcs in $F$. We show in the following that $M$ is 3-strong.


Observe that $M$ is clearly $3$-strong in $S$. Now consider some $x \in X$. Observe that $M$ contains the paths $s_x^{j}w_x^{1}$ for $j=1,2$ and the path $s_{(x,C_1)}q_{(x,C_1)}w_x^{1}$. Next, $M$ contains the paths $w_x^{1}s_x^{j}$ for $j=1,2$. Finally, if $\phi(x)=TRUE$, then $M$ contains the path $w_x^{1}q_{(x,C_1)}s_{(x,C_1)}$ and if $\phi(x)=FALSE$, then $M$ contains the path $w_x^{1}p_{(x,C_2)}s_{(x,C_2)}$. We obtain by Proposition \ref{connin} that $M$ is 3-strong in $S\cup w_x^1$. Similar arguments show that $M$ is 3-strong in $S'=S \cup \{w_x^{1},\ldots,w_x^4\}$.

Next observe that $M$ contains the paths $p_{(x,C_1)}s_{(x,C_1)},p_{(x,C_1)}w_x^4$ and $p_{(x,C_1)}q_{(x,C_1)}w_x^1$. Next, $M$ contains the path $s_{(x,C_1)}p_{(x,C_1)}$. Further, if $\phi(x)=TRUE$, then $M$ contains the paths $s_{C_1}v_{C_1}p_{(x,C_1)}$ and $w_x^1q_{(x,C_1)}p_{(x,C_1)}$ and if $\phi(x)=FALSE$, then $M$ contains the paths $w_x^4p_{(x,C_1)}$ and $v_{C_1}q_{(x,C_1)}p_{(x,C_1)}$. We obtain by Proposition \ref{connin} that $M$ is 3-strong in $S\cup p_{(x,C_1)}$. Similar arguments show that $M$ is 3-strong in $S \cup V_x$. As $x$ was chosen arbitrarily, we obtain that $M$ is 3-strong in $S''=S \cup \bigcup_{x \in X}V_x$.


Now consider some $C \in \mathcal{C}$ and let $x,x'$ be the variables such that $(x,C)$ and $(x',C)$ are incident pairs. We show that $M$ is 3-strong in $S''\cup v_C$. First observe that $M$ contains the path $s_Cv_C$ and paths of length 1 from $V_x$ and $V_{x'}$ to $v_C$. Next, observe that $M$ contains arcs $a_x,a_{x'}$ from $v_C$ to $V_x$ and $V_{x'}$, respectively. Further, if $C \in \mathcal{C}''$, then $M$ contains the edge $v_Cs_C$. Otherwise, we have $C \in \mathcal{C}'$, so $C$ is satisfied by one of $x$ and $x'$, say $x$. Then $M$ contains an edge linking $V_x$ and $v_C$ and the second endvertex of this edge in $V_x$ is distinct from the head of $a_x$. We obtain by Proposition \ref{connin} that $M$ is 3-strong.

As $M$ is obtained from $D$ by deorienting at most $k$ arcs, we obtain that $(D,k)$ is a positive instance of 3SDO.

\medskip

Now suppose that $(D,k)$ is a positive instance of 3SDO, so there is a set $F\subseteq A$ with $|F|\leq k$ such that the mixed graph $M$ that is obtained from $D$ by deorienting the arcs in $F$ is 3-strong.

\begin{claim}\label{p1}
Let $x \in X$ and let $C_1,C_2,C_3$ be the clauses containing $x$ or $\bar{x}$ in the ordering used in the construction. Then either $p_{(x,C_1)}w_x^4\in F$ or $\{q_{(x,C_1)}w_x^1, p_{(x,C_1)}v_{C_1}\}\subseteq F$.
\end{claim}
\begin{proof}
As $M$ is 3-strong, there is an edge entering $\{p_{(x,C_1)},q_{(x,C_1)}\}$ in $M-\{s_{(x,C_1)},v_{C_1}\}$, so $F$ contains one of the arcs $p_{(x,C_1)}w_x^4$ and $q_{(x,C_1)}w_x^1$. As $M$ is 3-strong, there is an edge entering $p_{(x,C_1)}$ in $M-\{s_{(x,C_1)},q_{(x,C_1)}\}$, so $F$ contains one of the arcs $p_{(x,C_1)}w_x^4$ and $p_{(x,C_1)}v_{C_1}$. 
\end{proof}

Similarly, we can prove the following claims:
\begin{claim}\label{p2}
Let $x \in X$ and let  $C_1,C_2,C_3$ be the clauses containing $x$ or $\bar{x}$ in the ordering used in the construction. Then either $q_{(x,C_2)}w_x^2\in F$ or $\{p_{(x,C_2)}w_x^1, q_{(x,C_2)}v_{C_2}\}\subseteq F$.
\end{claim}

\begin{claim}\label{p3}
Let $x \in X$ and let $C_1,C_2,C_3$ be the clauses containing $x$ or $\bar{x}$ in the ordering used in the construction. Then either $p_{(x,C_3)}w_x^2\in F$ or $\{q_{(x,C_3)}w_x^3, p_{(x,C_3)}v_{C_3}\}\subseteq F$.
\end{claim}

\begin{claim}\label{p4}
Let $x \in X$ and let  $C_1,C_2,C_3$ be the clauses containing $x$ or $\bar{x}$ in the ordering used in the construction. Then either $q_{x}w_x^4\in F$ or $\{p_{x}w_x^3, q_xs_x^4\}\subseteq F$.
\end{claim}
We further need the following simple observation.
\begin{claim}\label{w}
Let $x \in X$ and let  $C_1,C_2,C_3$  be the clauses containing $x$ or $\bar{x}$ in the ordering used in the construction. Then at least one of the arcs $q_{(x,C_1)}w_x^1$ and $p_{(x,C_2)}w_x^1$, at least one of the arcs $q_{(x,C_2)}w_x^2$ and $p_{(x,C_3)}w_x^2$, at least one of the arcs $q_{(x,C_3)}w_x^3$ and $p_{x}w_x^3$ and at least one of the arcs $q_{x}w_x^4$ and $p_{(x,C_1)}w_x^4$ is contained in $F$. 
\end{claim}
\begin{proof}
As $M$ is 3-strong, for $i=1,\ldots,4$, there is an edge leaving $w_x^{i}$ in $M-\{s_x^1,s_x^2\}$.
\end{proof}


The following intermediate result is crucial for defining a truth assignment.
\begin{claim}\label{princ}
For every $x \in X$, there is a set $F_x \subseteq F \cap (A(D[V_x])\cup \delta^+_D(V_x))$ with $|F_x|=6$ such that either $q_{(x,C_2)}v_{C_2}\notin F_x$ or $\{p_{(x,C_1)}v_{C_1},p_{(x,C_3)}v_{C_3}\}\cap F=\emptyset$.
\end{claim}
\begin{proof}
Let $B_x$ be the set of arcs in $F$ which are incident to $w_x^{i}$ for some $i \in \{1,\ldots,4\}$. By Claim \ref{w}, we obtain that $|B_x|\geq 4$. If $|B_x|\geq 6$, the statement trivially follows.

If $|B_x|=5$, observe that by Claim \ref{w}, $F$ contains an arc incident to $w_x^{i}$ for $i \in \{1,\ldots,4\}$. Hence at least one of the following arcs is not contained in $B_x: p_{(x,C_1)}w_x^4,q_xw_x^4, p_{(x,C_3)}w_x^2, q_{(x,C_2)} w_x^{2}$.
If $p_{(x,C_1)}w_x^4\notin B_x,$ then Claim \ref{p1} yields  $p_{(x,C_1)}v_{C_1}\in F$, hence $F_x=B_x\cup p_{(x,C_1)}v_{C_1}$ has the desired properties. If $q_xw_x^4\notin B_x,$ then Claim \ref{p4} yields  $q_xs_x^4 \in F$, hence $F_x=B_x\cup q_xs_x^4$ has the desired properties. If $ p_{(x,C_3)}w_x^2\notin B_x,$ then Claim \ref{p3} yields  $p_{(x,C_3)}v_{C_3}\in F$, hence $F_x=B_x\cup p_{(x,C_3)}v_{C_3}$ has the desired properties. If $ q_{(x,C_2)} w_x^{2}\notin B_x,$ then Claim \ref{p2} yields  $q_{(x,C_2)}v_{C_2}\in F$, hence $F_x=B_x\cup q_{(x,C_2)}v_{C_2}$ has the desired properties.


Now suppose that $|B_x|=4$.  By Claims \ref{p1} to \ref{w}, we obtain that either we have $B_x=\{p_{(x,C_1)}w_x^4,p_{(x,C_2)}w_x^1,p_{(x,C_3)}w_x^2,p_xw_x^3\}$ or $B_x=\{q_{(x,C_1)}w_x^1,q_{(x,C_2)}w_x^2,q_{(x,C_3)}w_x^3,q_xw_x^4\}$.\\
If $B_x=\{p_{(x,C_1)}w_x^4,p_{(x,C_2)}w_x^1,p_{(x,C_3)}w_x^2,p_xw_x^3\}$, then it follows from Claims \ref{p2} and \ref{p4} that $\{q_{(x,C_2)}v_{C_2},q_xs_x^2\}\subseteq F$. Hence $F_x=B_x\cup \{q_{(x,C_2)}v_{C_2},q_xs_x^2\}$ has the desired properties.
If $B_x=\{q_{(x,C_1)}w_x^1,q_{(x,C_2)}w_x^2,q_{(x,C_3)}w_x^3,q_xw_x^4\}$, then Claims \ref{p1} and \ref{p3} yield $\{ p_{(x,C_1)}v_{C_1}, p_{(x,C_3)}v_{C_3}\}\subseteq F$. Hence $F_x=B_x\cup \{ p_{(x,C_1)}v_{C_1}, p_{(x,C_3)}v_{C_3}\}$ has the desired properties.
\end{proof}

Observe that the sets $F_x$ which exist by Claim \ref{princ} are not necessarily unique, however this ambiguity will not have any effect. Let $F_1=\bigcup_{x \in X}F_x$ and $F_2=F-F_1$. As the sets $F_x$ are pairwise disjoint, we have $|F_1|=6|X|$.

We now define a truth assignment $\phi:X \rightarrow \{TRUE,FALSE\}$ in the following way: We set $\phi(x)=TRUE$ if $q_{(x,C_2)}v_{C_2}\notin F$ and $\phi(x)=FALSE$ otherwise. We let $\mathcal{C}'$ be the sets of clauses in $\mathcal{C}$ such that an arc of $F_1$ enters $v_C$ and let $\mathcal{C}''=\mathcal{C}-\mathcal{C}'$. Observe that by construction, we have that $C$ is satisfied by $\phi$ for all $C \in \mathcal{C}'$. Further, observe that, as $M$ is 3-strong, for every $C \in \mathcal{C}''$, $F$ contains an arc $a_C$ entering $v_C$. By definition of $\mathcal{C}''$, we have $a_C \in F_2$. This yields $|\mathcal{C}''|\leq |F_2|$. We obtain $|\mathcal{C}'|=|\mathcal{C}|-|\mathcal{C}''|\geq |\mathcal{C}|-|F_2|=|\mathcal{C}|-|F|+|F_1|\geq |\mathcal{C}|-k+6|X|=|\mathcal{C}|-(6|X|+|\mathcal{C}|-\ell)+6|X|=\ell$. As $\phi$ satisfies all clauses in $\mathcal{C}'$, we obtain that $(X,\mathcal{C},\ell)$ is a positive instance of S3BMAX2SAT.

This finishes the proof.
\end{proof}
Observe that there are two canonical optimization problems associated to 3SDO. Firstly, we can minimize the number of arcs we deorient and secondly, we can maximize the number of arcs we do not deorient. We wish to remark that APX-hardness results for both these optimization problems can be obtained by combining the reduction proving Theorem \ref{3sdohard} with the corresponding result in \cite{bk}.

\subsection{Deorienting to increase in- and out-degrees}\label{degrees}
We here prove a result on deorienting arcs to satisfy certain degree conditions. 
I deleted the above definition as the same notation was used for two different things.
\begin{proposition}
  There exists a polynomial algorithm for the following problem: given a digraph $D=(V,A)$ and a natural number $k$; find
  a minimum subset of arcs in $A$ whose deorientation leads to a mixed graph $M$ in which $\min\{d_M^+(v)+d_M(v),d_M^-(v)+d_M(v)\}\geq k$ holds for all $v \in V$.
\end{proposition}
\begin{proof}
Let $D=(V,A)$ and $k$ be given.
  
Form a flow network ${\cal N}(D)$  with vertex set  $V_1\cup V_2\cup\{s,t\}$, where for $i=1,2, V_i$ contains a copy $v_i$ of every $v \in V$,  
and arc-set $\{sv_1|v_1\in V_1\}\cup\{v_2t|v_2\in V_2\}\cup{} A_1\cup A_2$, where $A_1=\{u_1v_2|uv\in A\}$ and $A_2=\{v_1u_2|uv\in A\}$. For $i=1,2$, let all arcs of $A_1\cup A_2$ have capacity  1, lower bound 0 and cost $i-1$. Finally let all arcs starting in $s$ or ending in $t$ have lower bound $k$, capacity $k|V|$ and cost zero.
  Now it is easy to check that a feasible integer valued $(s,t)$-flow of cost $C$ in ${\cal N}(D)$ corresponds to a set of $C$ arcs from $A_2$ whose deorientation results in a mixed graph $M$ satisfying $\min\{d_M^+(v)+d_M(v),d_M^-(v)+d_M(v)\}\geq k$ for all $v \in V$, and conversely. Hence we can find a set of arcs with the desired property of minimum size  by
  finding a minimum cost feasible integer valued $(s,t)$-flow in ${\cal N}(D)$ which is well-known to be possible in polynomial time, see e.g. \cite[Chapter 4]{BG}.
  \end{proof}

\subsection{Deorienting to obtain specified local arc-connectivities}\label{local}

In this section, we deal with a problem concerning deorientation for local arc connectivities. The reduction establishes an interesting connection between deorientations and orientations. We need to consider the following orientation problem:

\medskip

\begin{tabular}{|ccc|}\hline
 & \begin{minipage}{14cm}
\vspace{2mm}

\noindent{}{\bf Local Connectivity Orientation (LCO)}
\vspace{2mm}

{\bf Input:} A graph $G$, a requirement function $r:V(G)\times V(G)\rightarrow \mathbb{Z}_{\geq 0}$.

\vspace{2mm}

{\bf Question:} Is there an orientation $\vec{G}$ of $G$ that satisfies $\lambda_{\vec{G}}(x,y)\geq r(x,y)$ for all $x,y \in V(G)$?
\medskip

\vspace{2mm}

\end{minipage} & \\ \hline
\end{tabular}
\medskip

We use the following result of Frank, Kir\'aly and Kir\'aly \cite{fkk}, see also \cite{h}:

\begin{proposition}\label{lcohard}\cite{fkk}
LCO is NP-hard.
\end{proposition}

Actually, we need the following slight strengthening of Proposition \ref{lcohard}.
\begin{proposition}\label{lcohard1}
LCO is NP-hard even for instances $(G,r)$ with $r(x,y)\geq 1$ for all all $x,y \in V(G)$.
\end{proposition}

\begin{proof}
    We prove this by a reduction from LCO. Let $(G,r)$ be an instance of LCO. We now create a graph $G'$ from $G$ by adding two vertices $a$ and $b$, an edge $ab$ and edges $ax$ and $bx$ for all $x \in V(G)$. Further, we define $r':V(G')\times V(G')\rightarrow \mathbb{Z}_{\geq 0}$ by $r'(ab)=|V(G)|,r'(b,a)=1$, $r'(x,a)=r'(a,x)=r'(x,b)=r'(b,x)=1$ for all $x \in V(G)$ and $r'(x,y)=r(x,y)+1$ for all $x,y \in V(G)$.

    Observe that $r'(x,y)\geq 1$ for all $x,y \in V(G')$. We now show that $(G',r')$ is a positive instance of LCO if and only if $(G,r)$ is a positive instance of LCO.

   First suppose that $(G',r')$ is a positive instance of LCO, so there is an orientation $\vec{G'}$ of $G'$ for which $\lambda_{\vec{G'}}(x,y)\geq r'(x,y)$ holds for all $x,y \in V(G')$. As $\lambda_{\vec{G'}}(a,b)\geq r'(a,b)\geq 1$ and $\lambda_{\vec{G'}}(b,a)\geq r'(b,a)\geq 1$, we obtain that $\vec{G'}$ contains a directed cycle which contains the arc corresponding to the edge $ab$. Possibly reversing the orientation of all arcs of this cycle, we may suppose that the edge $ab$ is oriented as $ba$ in $\vec{G'}$. As $\lambda_{\vec{G'}}(a,b)\geq r'(a,b)=|V(G)|$, we obtain that for all $x \in V(G)$, the edge $xa$ is oriented as $ax$ and the edge $xb$ is oriented as $xb$ in $\vec{G'}$. Let $\vec{G}=\vec{G'}[V(G)]$. Observe that $\vec{G}$ is an orientation of $G$. Further, for all $x,y \in V(G)$, we have $\lambda_{\vec{G}}(x,y)=\lambda_{\vec{G'}}(x,y)-1\geq r'(x,y)-1=r(x,y)$. Hence $(G,r)$ is a positive instance of LCO.

    Now suppose that $(G,r)$ is a positive instance of LCO, so there is an orientation $\vec{G}$ of $G$ with $\lambda_{\vec{G}}(x,y)\geq r(x,y)$ for all $x,y \in V(G)$. Let an orientation $\vec{G'}$ of $G'$ be obtained by orienting the edge $ab$ as $ba$, orienting the edge $xa$ as $ax$ and orienting the edge $xb$ as $xb$ for all $x \in V(G)$ and giving all other edges the orientation they have in $\vec{G}$. Clearly, we have $\lambda_{\vec{G'}}(a,b)=|V(G)|=r'(a,b), \lambda_{\vec{G'}}(b,a)=1=r'(b,a)$ and $\lambda_{\vec{G'}}(a,x)\geq 1=r'(a,x),\lambda_{\vec{G'}}(x,a)\geq 1=r'(x,a),\lambda_{\vec{G'}}(b,x)\geq 1=r'(b,x)$ and $\lambda_{\vec{G'}}(x,b)\geq 1=r'(x,b)$ for all $x \in V(G)$. Further, for all $x,y \in V(G)$, we have $\lambda_{\vec{G'}}(x,y)=\lambda_{\vec{G}}(x,y)+1\geq r(x,y)+1=r'(x,y)$. Hence $(G',r')$ is a positive instance of LCO.
\end{proof}


For the main result of this section, we formally consider the following deorientation problem:
\medskip

\begin{tabular}{|ccc|}\hline
 & \begin{minipage}{14cm}
\vspace{2mm}

\noindent{}{\bf Local Connectivity DeOrientation (LCDO)}
\vspace{2mm}

{\bf Input:} A digraph $D$, a requirement function $r:V(D)\times V(D)\rightarrow \mathbb{Z}_{\geq 0}$, an integer $k$.

\vspace{2mm}

{\bf Question:} Is there a mixed graph $M$ that can be obtained from $D$ by deorienting at most $k$ arcs and which satisfies $\lambda_{M}(x,y)\geq r(x,y)$ for all $x,y \in V(D)$?
\medskip

\vspace{2mm}

\end{minipage} & \\ \hline
\end{tabular}
\medskip

We prove the following:

\begin{theorem}\label{lcdohard}
LCDO is NP-hard.
\end{theorem}
\begin{proof}
    We prove this by a reduction from the restriction of LCO described in Proposition \ref{lcohard1}.  Let $(G,r)$ be an instance of LCO with $r(u,v)\geq 1$ for all $u,v \in V(G)$. We now obtain a digraph $D$ by replacing every $e=uv \in E(G)$ by a new vertex $w_e$ and the arcs $uw_e$ and $vw_e$. We further define $r':V(D)\times V(D)$ by\\
    \begin{center}
    $r'(x,y)=
\begin{cases}
r(x,y), \text{ if }x,y \in V(G)\\
1, \text{ if }x \in V(D)-V(G) \text{ and }y \in V(G)\\
0, \text{ otherwise. }
\end{cases}$
\end{center}

Finally, we let $k=|E(G)|$. We show in the following that $(D,r',k)$ is a positive instance of LCDO if and only if $(G,r)$ is a positive instance of LCO. 

First suppose that $(G,r)$ is a positive instance of LCO, so there is an orientation $\vec{G}$ of $G$ that satisfies $\lambda_{\vec{G}}(x,y)\geq r(x,y)$ for all $x,y \in V(G)$. We now obtain the mixed graph $M$ from $D$ by deorienting the arc $vw_e$ for all $e=uv \in E(G)$ which are oriented as $uv$ in $\vec{G}$. It is easy to see that $\lambda_{M}(x,y)=\lambda_{\vec{G}}(x,y)\geq r(x,y)=r'(x,y)$ holds for all $x,y \in V(G)$. Now let $x \in V(G)$ and $e=uv \in E(G)$ such that $e$ is oriented to $uv$ in $\vec{G}$. Then we have $\lambda_{M}(w_e,x)\geq \min\{\lambda_{M}(w_e,v),\lambda_{M}(v,x)\}\geq \min\{1,r(v,y)\}\geq 1$. As $M$ is obtained from $D$ deorienting $k$ arcs, we obtain that $(D,r',k)$ is a positive instance of LCDO.

Now suppose that $(D,r',k)$ is a positive instance of LCDO, so there is a set $F\subseteq A(D)$ of at most $k$ arcs such that the mixed graph $M$ obtained from $D$ by deorienting these edges satisfies $\lambda_{M}(x,y)\geq r'(x,y)$ for all $x,y \in V(D)$. For every $e=uv \in E(G)$, as $\lambda_{M}(w_e,u)\geq r'(w_e,u)=1$, we obtain that at least one of the edges $uw_e$ and $vw_e$ is contained in $F$. As $k=|E(G)|$, exactly one of these edges is contained in $F$. Now let an orientation $\vec{G}$ of $G$ be obtained by orienting every edge $e=uv \in E(G)$ as $uv$ when $vw_e \in F$. We then have $\lambda_{\vec{G}}(x,y)=\lambda_{M}(x,y)\geq r'(x,y)=r(x,y)$ for all $x,y \in V(G)$, so $(G,r)$ is a positive instance of LCO.
\end{proof}

\subsection{Approximating \boldmath$deor_k^{arc}(D)$}\label{appro2}

As mentioned earlier the number $deor_k^{arc}(D)$ can be found in polynomial time for $k=1$. For all other values of $k$ the complexity status is open, see Problem \ref{probarc}.

On the other hand, the number $deor^{arc}_k(D)$ can be found in polynomial time for every $k$ when $D$ is a tournament (that is, an orientation of a complete graph) as proven by the first author and Yeo in \cite{bangDAM136}. We show below how to obtain a 2-approximation for all $k$ and every digraph $D$.\\

An {\bf out-branching} ({\bf in-branching}) rooted at the vertex $s$ in a digraph $D$, denoted $B_s^+$ ($B^-_s$),
is a connected spanning subdigraph of $D$ in which the in-degree (out-degree) of every vertex except $s$ is 1 and $s$ has in-degree (out-degree) 0.

\begin{theorem}[Edmonds]\cite{edmonds1973}\label{edmo}
  A digraph $D$ has $k$ arc-disjoint out-branchings rooted at the vertex $s\in V$ if and only if
  \begin{equation}
    \label{kbranchcond}
    d_D^+(X)\geq k\text{ for all } X \subsetneq V(D)\text{ with }s \in X
    \end{equation}
  \end{theorem}

We further need the following result, due to Edmonds, which is based on matroidal methods, 
see e.g. \cite[Corollary 53.10a]{schrijver2003}.

\begin{proposition}\label{matro}
    Let $D=(V,A)$ be a digraph, $s \in V$, $k$ a positive integer and $w:A \rightarrow \mathbb{R}_{\geq 0}$ a nonnegative weight function. Then a packing $\mathcal{B}$ of $k$ out-branchings (in-branchings) rooted at $s$ in $D$ that minimizes $w(A(\mathcal{B}))$ can be found in polynomial time if such a packing exists.
\end{proposition}

We are now ready to prove the following theorem which is the main result of this section:
\begin{theorem}
 There is an algorithm whose input is a digraph $D$ and an integer $k$ and that returns an arc set $F\subseteq A(D)$ whose deorientation results in a $k$-arc-strong mixed graph and is at most twice as big as a smallest set with this property.
  \end{theorem}
  \begin{proof}
      Let the digraph $D'$ be obtained by adding an arc $a'=vu$ for every arc $a=uv \in A(D)$. We further define a weight function $w:A(D')\rightarrow \mathbb{R}_{\geq 0}$ by $w(a)=0$ and $w(a')=1$ for all $a \in A(D)$. We further fix some arbitary vertex $s \in V(D)$. By Proposition \ref{matro}, we can now calculate a packing $\mathcal{B}_1$ of $k$ out-branchings rooted at $s$ that minimizes $w(A(\mathcal{B}_1))$ and a packing $\mathcal{B}_2$ of $k$ in-branchings rooted at $s$ that minimizes $w(A(\mathcal{B}_2))$ in polynomial time. Let $F$ be the set of arcs $a \in A(D)$ such that $a' \in A(\mathcal{B}_1)\cup A(\mathcal{B}_2)$ and let the algorithm output $F$. It follows from Theorem \ref{edmo} that $F$ is a feasible solution for the deorientation problem. Now let $F^*$ be an optimal solution to the deorientation problem. It follows from Theorem \ref{edmo} that $A(D)\cup \{a':a \in F^*\}$ contains a set of $k$ arc-disjoint out-branchings rooted at $s$ and a set of $k$ arc-disjoint in-branchings rooted at $s$. This yields $|F|=w(A(\mathcal{B}_1) \cup A(\mathcal{B}_2))\leq w(A(\mathcal{B}_1))+w( A(\mathcal{B}_2))\leq 2|F^*| $, hence $F$ has the desired properties.
  \end{proof}

\bibliographystyle{alpha}

\end{document}